 \documentclass[reqno,11pt,a4paper]{amsart}

\usepackage{pgfplots}
\usepackage{subfigure}
\pgfplotsset{compat=1.12} %\pgfplotsset{compat=1.16}
\usepackage{listings}

\usepackage{amsfonts,amsmath,amssymb,amsthm,mathtools}
\usepackage[latin1]{inputenc}
\usepackage{dsfont}
\usepackage{enumerate}
\usepackage{xcolor}
\usepackage[all]{xy}
\def\notshow#1\notshowend{} %
\usepackage{hyphenat}

\usepackage{graphicx}

\usepackage{tikz}
\usepackage{tikz-3dplot}
\usepackage{pgfplots}

\usepackage{multirow}
\usepackage{array}

\usepackage{amssymb}
\usepackage{amsfonts}
\usepackage{mathrsfs}
\usepackage{hyperref}

\usepackage[normalem]{ulem} % PARA TACHAR

\def\bb#1\eb{\textcolor{blue}{#1}}
\def\br#1\er{\textcolor{red}{#1}} %
\def\bm#1\em{\textcolor{purple}{#1}} %

\newtheorem{thm}{Theorem}[section]
\newtheorem{prop}[thm]{Proposition}
\newtheorem{lemma}[thm]{Lemma}

\theoremstyle{definition}
\newtheorem{defi}[thm]{Definition}
\newtheorem{notation}[thm]{Notation}
\newtheorem{exe}[thm]{Example}

\newtheorem{rem}[thm]{Remark}

\newtheorem*{theorem*}{Proposition}

\newcommand{\ben}{\begin{enumerate}}
\newcommand{\een}{\end{enumerate}}
\newcommand{\bit}{\begin{itemize}}
\newcommand{\eit}{\end{itemize}}
\newcommand{\edoc}{\end{document}}

\title[Snell's law revisited]{Snell's law revisited and generalized \\ via Finsler Geometry}

\author[S. Markvorsen]{Steen Markvorsen}\address{DTU Compute, Mathematics, \hfill\break\indent Technical University of Denmark, \hfill\break\indent DK-2800 Kgs. Lyngby, Denmark.}\email{stema@dtu.dk}

\author[E. Pend\'{a}s-Recondo]{Enrique Pend\'{a}s-Recondo}\address{Departamento de Matem\'{a}ticas, \hfill\break\indent Universidad de Murcia, \hfill\break\indent Campus de Espinardo, \hfill\break\indent 30100 Espinardo, Murcia, Spain.}\email{e.pendasrecondo@um.es}

\subjclass{53B40, 53C60, 53C22, 53Z05, 78M30}

\begin{document}

\begin{abstract}
We study the variational problem of finding the fastest path between two points that belong to different anisotropic media, each with a prescribed speed profile and a common interface. The optimal curves are Finsler geodesics that are refracted -- broken -- as they pass through the interface, due to the discontinuity of their velocities. This ``breaking'' must satisfy a specific condition in terms of the Finsler metrics defined by the speed profiles, thus establishing the generalized Snell's law. In the same way, optimal paths bouncing off the interface -- without crossing into the second domain -- provide the generalized law of reflection. The classical Snell's and reflection laws are recovered in this setting when the velocities are isotropic. If one considers a wave that propagates in all directions from a given ignition point, the trajectories that globally minimize the traveltime generate the wavefront at each instant of time. We study in detail the global properties of such wavefronts in the Euclidean plane with anisotropic speed profiles. Like the individual rays, they break when they encounter the discontinuity interface. But they are also broken due to the formation of cut loci -- stemming from the self-intersection of the wavefronts -- which typically appear when they approach a high-speed profile domain from a low-speed profile.

\vspace{10mm}

\noindent {\em Keywords}: Snell's law, Fermat's principle, Huygens' principle, Zermelo navigation, Finsler metrics, anisotropic discontinuous media, rays with least traveltime, wave propagation.
\end{abstract}

\maketitle

\pagebreak

\tableofcontents

\section{Introduction}

Snell's law is -- in physical terms -- concerned with the refraction of light which in space passes from a point $ q_{1} $  to another point $ q_{2} $ through an interface between two (usually isotropic) media $ Q_{1} $ and $ Q_{2} $ with different refraction indices $ n_{1} $ and $ n_{2} $, say. In its simplest form the law expresses the break of the light path at the interface in terms of the incoming angle, the outgoing angle, and the ratio between the two refraction indices. The light is supposed to have a constant speed $ c/n_{1} $ in all directions in $ Q_{1} $ and similarly $ c/n_{2} $ in $ Q_{2} $ (where $ c $ denotes the speed of light in vacuum). For convenience of the reader, this classical formula is presented and discussed in Examples~\ref{ex:classical_snell} and \ref{exe:Classic2} below -- in the more general framework of the present paper.

Snell's law is classically derived from Fermat's principle, which states that the path of the light is the one which takes the least time to go from $ q_{1} $ to $ q_{2} $. This principle will also be our starting point. It motivates directly the variational approach to the task of finding optimal traveltimes -- also through the interface of \emph{anisotropic} inhomogeneous domains which is the main concern of the present paper. The anisotropy is here given by having different prescribed speed profiles $ V_{1} $ and $ V_{2} $ in the two domains. These speed profiles, that we now allow to depend on both the position and the \emph{direction} of the travel path, are considered in a given Riemannian manifold $ (Q, g) $ which thus defines the background for the domains $ Q_{1} \subset Q $, $ Q_{2} \subset Q $,  their interface $ \eta \subset Q $ as well as the space of paths under consideration.

We then show how the traveltime through the interface $ \eta $ in this general situation is minimized by broken geodesics of two Finsler metrics -- one on each domain -- which are determined by the ambient metric $g$ and the prescribed speed profiles $ V_{i} $. And we show and illustrate how the ``breaking'' at the interface must satisfy a specific condition on these Finsler metrics, i.e. the generalized Snell's law. We note that the classical Snell's law (with isotropic constant speed profiles   $ c/n_{i} $ in $ (Q,g) $) of course also follows from the general version -- but then with  simpler Finsler metrics which are just the conformally modified background metrics $ \rho^{2}_{i}\cdot g $, where the conformal factors are $ \rho_{i} = n_{i}/c $.

Once the full set of relevant (possibly broken) time-parametrized Finsler geodesics issuing from one point $ q_{1} $ has been found, it gives rise to a foliation of the domain $ Q $ by the wavefronts, each consisting of those points which have the same traveltime from $ q_{1} $ -- assuming that the geodesic equations have solutions for sufficiently large time. Some geodesics are not globally time-minimizing, so they do not contribute to the wavefront. This means that they have encountered a cut locus on their way from $ q_{1} $. We discuss in detail, and illustrate via a simple example (with constant speed profiles  $ V_{i} $ in $ \mathds{R}^{2} $), how such cut loci typically appear in situations where $ V_{2} $ is ``larger'' than $ V_{1} $. In this way -- already in this example -- the creation of cut loci follows from the discontinuous speed profile in $ Q $ notwithstanding the fact that there is neither curvature nor structural topology present in neither $ Q_{1} $ nor $ Q_{2} $.

\subsection*{Outline of the paper}
We begin in \S~\ref{sec:initial_concepts} with the equivalence between the time spent when following a curve with a prescribed anisotropic velocity and the length of the curve computed with a certain Finsler metric. As alluded to above, when the velocity changes discontinuously from one medium to another, we have a situation that generalizes the classical Snell's law. In \S~\ref{sec:time_min} we apply calculus of variations to find the critical points of the traveltime functional, i.e. the refracted trajectories, thus arriving at the generalized Snell's law \eqref{eq:snell}. The precise results are summarized in Theorem~\ref{th:refraction}. Following an analogous procedure, we derive in \S~\ref{sec:reflection} a similar result, Theorem~\ref{th:reflection}, for the reflected trajectories, including the generalized law of reflection \eqref{eq:reflection}. Finally, in \S~\ref{sec:R2} we reduce the setting to $ \mathds{R}^2 $ with constant Finsler metrics (i.e. Minkowski norms). This allows us to address some global issues, such as the existence and uniqueness of solutions to Snell's law and the law of reflection, which are formulated in Theorems~\ref{th:existence_refraction} and \ref{th:existence_reflection}, as well as the study of the global minima of the traveltime functional. When considering rays that are shot in all directions from a given point of ignition, the trajectories that globally minimize the traveltime are the ones that generate the wavefront at each given time. This phenomenon is described and illustrated in Theorem~\ref{th:wavefront}. When a wavefront overlaps itself (at the intersection of two rays) it forms the cut locus, which is studied and characterized for a specific example in \S~\ref{subsec:cutlocus}.

\section{Initial concepts}
\label{sec:initial_concepts}
\subsection{Traveltime and Finsler metrics}
Let $ (Q,g) $ be a smooth Riemannian manifold of dimension $ n \geq 2 $, and consider also its tangent bundle $ TQ $. The metric $ g $ will measure actual distances in the space $ Q $, i.e. if $ \gamma: [0,t_0] \rightarrow Q $ is a (piecewise smooth) curve with tangent vector $ \gamma'(t) $, then the distance traveled -- when following this trajectory -- is
\begin{equation}
\label{g_length}
L_g[\gamma] \coloneqq \int_{0}^{t_0} \sqrt{g(\gamma'(t),\gamma'(t))} dt = \int_{0}^{t_0} \|\gamma'(t)\|_{g} \, dt
\end{equation}
(splitting the integral accordingly at any break). This distance, this length, is independent of the parametrization of $ \gamma $.

Assume now that we travel along $ \gamma $ with a prescribed velocity. We will allow this velocity to be anisotropic, i.e. it may depend on the point of the space and also on the direction. This means that there is a function
\begin{equation*}
V: TQ \setminus \textbf{0} \rightarrow (0,\infty)
\end{equation*}
such that $ V(p,v) $ represents a fixed speed at $ p \in Q $ in the direction $ v \in T_pQ \setminus \{0\} $. Since $ V(p,v) $ depends on the (oriented) direction of $ v $ but not on its length, this function must be positively homogeneous of degree zero, i.e. $ V(p,v) = V(p,\lambda v) $ for all $ \lambda > 0 $. Then the time spent traveling along $ \gamma $ with this prescribed velocity is
\begin{equation}
\label{eq:time}
T[\gamma] \coloneqq \int_{0}^{t_0} \frac{\|\gamma'(t)\|_{g}}{V(\gamma(t),\gamma'(t))} \,dt,
\end{equation}
which is again independent of reparametrizations (as long as they preserve the orientation of $ \gamma $) thanks to the zero-homogeneity of $ V $. The preservation of orientation is necessary, since in general we will not assume that  $ V(p,v) = V(p,-v) $. Nevertheless, observe that if $ \gamma $ is parametrized in such a way that its tangent vector always coincides with the velocity given by $ V $, i.e. $ \|\gamma'(t)\|_g = V(\gamma(t),\gamma'(t)) $, then the parameter $ t $ directly measures the time, as $ T[\gamma] = t_0 $. We say in this case that $ \gamma $ is {\em time-parametrized}.

Let us denote by $ F $ the function given by the integrand in \eqref{eq:time}, i.e.
\begin{equation*}
\begin{array}{cccc}
F \colon & TQ & \longrightarrow & [0,\infty) \\
& (p,v) & \longmapsto & F(p,v) := \frac{\|v\|_{g}}{V(p,v)},
\end{array}
\end{equation*}
with $ F(p,0) := 0 $ for all $ p \in Q $. Then $ F $ is a Finsler metric if and only if $ V $ is smooth and the following oval (the indicatrix of $ F $), that it defines at each $ p \in Q $ is a strongly convex hypersurface of $ T_pQ $:\footnote{Using any Euclidean metric in $ T_pQ $, this means that the second fundamental form of $ \Sigma_p $ with respect to the inner normal vector is positive definite or, equivalently, that $ \Sigma_p $ has positive sectional curvature everywhere (see \cite[Prop.~2.3]{JS1}).}
\begin{equation*}
\Sigma_p := \{v \in T_pQ: \|v\|_{g} = V(p,v)\} = \{v \in T_pQ: F(p,v) = 1 \}.
\end{equation*}
Note that the length of $ \gamma $ computed with $ F $ coincides with the traveltime along $ \gamma $ with the prescribed velocity $ V $:
\begin{equation}
\label{F_length}
L_F[\gamma] := \int_{0}^{t_0} F(\gamma(t),\gamma'(t)) dt = T[\gamma].
\end{equation}

\subsection{Snell's law: The general setting}
In the classical Snell's law setting, a light ray travels between two different regions in $ \mathds{R}^2 $ separated by a straight line. Let us generalize this framework by having two open subsets $ Q_1, Q_2 \subset Q $ such that their closures $ \overline{Q}_1, \overline{Q}_2 $ are smooth manifolds with the same smooth boundary $ \eta := \partial Q_1 = \partial Q_2 $ and $ \overline{Q}_1 \cup \overline{Q}_2 = Q $. Assume also that we have two different velocity functions -- speed profiles -- $ V_1, V_2 $ defined on $ TQ_1 $ and $ TQ_2 $, respectively, and let $ F_1, F_2 $ be their corresponding Finsler metrics, so that $ (Q_1,F_1) $ and $ (Q_2,F_2) $ are smooth Finsler manifolds (and $ (Q,g) $ is still a Riemannian one).

Then, if $ \gamma:[0,t_0] \rightarrow Q $ is a curve going from $ \gamma(0) = q_1 \in Q_1 $ to $ \gamma(t_0) = q_2 \in Q_2 $ and crossing $ \eta $ once at $ \gamma(\tau) $, for some $ \tau \in (0,t_0) $, we can still compute the traveltime by
\begin{equation}
\label{eq:traveltime}
T[\gamma] = \int_{0}^{\tau} F_1(\gamma(t),\gamma'(t))dt + \int_{\tau}^{t_0} F_2(\gamma(t),\gamma'(t))dt,
\end{equation}
where, formally, $ \int_{0}^{\tau} := \lim_{h\rightarrow \tau^-} \int_{0}^{h} $ and $ \int_{\tau}^{t_0} := \lim_{h\rightarrow \tau^+} \int_{h}^{t_0} $. To avoid integrability and differentiability issues, we will further assume that $ F_1 $ and $ F_2 $ can be smoothly extended to $ \eta $, so that the integrals always exist and we can also compute $ F_1(\gamma(\tau),\gamma'(\tau^-)) $, $ F_2(\gamma(\tau),\gamma'(\tau^+)) $ and their derivatives.

\begin{defi}
\label{def:traveltime}
Fixing an interval $ [0,t_0] $, some $ \tau \in (0,t_0) $ and two points $ q_1 \in Q_1, q_2 \in Q_2 $, let $ \mathcal{N} $ be the set of all (regular) piecewise smooth curves $ \gamma:[0,t_0] \rightarrow Q $ from $ \gamma(0) = q_1 $ to $ \gamma(t_0) = q_2 $ and crossing $ \eta $ once at $ \gamma(\tau) $. We will call $ T_{\mathcal{N}}: \mathcal{N} \rightarrow \mathds{R} $ defined by \eqref{eq:traveltime} the {\em traveltime functional} on $ \mathcal{N} $.

Note that any curve from $ q_1 $ to $ q_2 $ (piecewise smooth and crossing $ \eta $ once) can be reparametrized so that it belongs to $ \mathcal{N} $, without affecting its traveltime. Therefore, the choice of $ t_0 $ and $ \tau $ is arbitrary and non-restrictive (as well as the choice of $ 0 $ as the initial time).
\end{defi}

Among all the possible paths in $ \mathcal{N} $, light (or any wave, in general) will choose the one that minimizes the traveltime fuctional, according to the classical Fermat's principle. In a smooth Finsler manifold $ (Q,F) $, it is well known that this trajectory (if it exists) must be a geodesic of $ F $ (see e.g. \cite[\S~5]{BCS}). In our setting, we will have an $ F_1 $-geodesic from $ q_1 $ to $ \eta $ (the incident trajectory) and an $ F_2 $-geodesic from $ \eta $ to $ q_2 $ (the refracted trajectory), with a break at the point where the curve meets $ \eta $. The goal is to find the change in the direction at the break point, i.e. the generalized Snell's law.

\begin{rem}
The functions $ V_1 $ and $ V_2 $ can represent not only the propagation velocities of a wave, but also the maximum possible velocities of an object or a person in two different media. In the former case, the problem is related to Fermat's principle (see e.g. \cite{BS,CJM,CJS1,GHWW}) and Huygens' principle (see e.g. \cite{DS,JP,JPS1,M2}), with real-world applications in seismic theory \cite{ABS,YN}, sound waves \cite{GW} or wildfire modeling \cite{D,G,JPS2,M1}, among others. In the latter, it becomes a general version of Zermelo's navigation problem, which seeks the fastest trajectory between two points for a moving object in a medium which, in turn, may also be moving with respect to the observer (see e.g. \cite{CJS2,JS2}). Thus, real-world applications of this generalized Snell's law are multiple and widely varied. Some examples are naturally encountered when studying for example:
\begin{itemize}
\item The trajectories of light rays between two anisotropic media.
\item The firefront of a wildfire that spreads over two different crop or fuel fields.
\item The propagation of seismic waves when passing through layers of different anisotropic inhomogeneous materials.
\item The fastest trajectory between a point on the ground and a point in the sea for a person who first has to run and then swim, taking the water current into account.
\item The fastest trajectory for a travel first by boat and then by zeppeling, taking into account both the water current and the wind.
\end{itemize}
\end{rem}

\section{Time-minimizing trajectories}
\label{sec:time_min}
\subsection{Variational definitions and conventions}
In order to find  time-minimizing trajectories, one looks for the critical points of $ T_{\mathcal{N}} $. Let $ \gamma \in \mathcal{N} $ be our candidate curve, with $ \gamma(0) = q_1 \in Q_1 $, $ \gamma(t_0) = q_2 \in Q_2 $ and $ \gamma(\tau) \in \eta $. Recall from Definition~\ref{def:traveltime} that the choice of $ t_0 $ and $ \tau $ is arbitrary, so we can further assume, without loss of generality, that $ \gamma $ is time-parametrized, i.e.
\begin{equation}
\label{eq:time_par}
\begin{split}
& F_1(\gamma(t),\gamma'(t)) = 1, \qquad \forall t \in [0,\tau), \\
& F_2(\gamma(t),\gamma'(t)) = 1, \qquad \forall t \in (\tau,t_0],
\end{split}
\end{equation}
(so $ F_1(\gamma(\tau),\gamma'(\tau^-)) = 1 $ and $ F_2(\gamma(\tau),\gamma'(\tau^+)) = 1 $). The curve $ \gamma $ will be the base curve in a family of curves defined by a variation, in the following sense.

\begin{defi}
\label{def:variation}
An {\em admissible variation} $ H $ of $ \gamma $ is a continuous map
\begin{equation*}
\begin{array}{cccc}
H \colon & (-\varepsilon,\varepsilon) \times [0,t_0] & \longrightarrow & Q \\
& (\omega,t) & \longmapsto & H(\omega,t) := \gamma_{\omega}(t)
\end{array}
\end{equation*}
such that:
\begin{itemize}
\item $ H $ is smooth on each $ (-\varepsilon,\varepsilon) \times [s_i,s_{i+1}] $, where $ 0 = s_1 < \ldots < s_k = \tau < \ldots < s_{k+l} = t_0 $ is a partition of $ [0,t_0] $.\footnote{For simplicity, we will only focus on the break $ s_k=\tau $ and omit writing the others explicitly, as they are not relevant for our study here.}

\item $ H(\omega,\cdot) = \gamma_{\omega} \in \mathcal{N} $ for all $ \omega \in (-\varepsilon,\varepsilon) $, with $ \gamma_0 = \gamma $. In particular, $ \gamma_{\omega}(0) = q_1 $, $ \gamma_{\omega}(t_0) = q_2 $ and $ \gamma_{\omega}(\tau) \in \eta $.

%\item $ H(0,t) = \gamma_{0}(t) = \gamma (t) $ for all $ t \in [0,t_0] $.

%\item For each $ t \in [0,t_0] $, there exists a positive $ \varepsilon_t \leq \varepsilon $ such that $ H(\cdot,t) $ is smooth on $ (-\varepsilon_t,\varepsilon_t) $.
\end{itemize}

The first condition ensures that the tangent vector of the curve $ \omega \mapsto \gamma_{\omega}(t) $ at $ \omega = 0 $ is always well-defined. Let us denote by $ W(t) \in T_{\gamma(t)}Q $ this tangent vector. Then the curve in $ TQ $ given by
\begin{equation*}
\begin{array}{ccc}
[0,t_0] & \longrightarrow & TQ \\
t & \longmapsto & (\gamma(t),W(t))
\end{array}
\end{equation*}
is called the {\em variational field} of $ \gamma $ associated with $ H $.

Finally, we say that $ \gamma $ is a {\em critical point} of $ T_{\mathcal{N}} $ if, for every admissible variation $ H $, we have that
\begin{equation*}
\left. \frac{d}{d\omega} T_{\mathcal{N}}[\gamma_{\omega}] \right\rvert_{\omega=0} = 0.
\end{equation*}
\end{defi}

\begin{rem}
Note that the only time-parametrized curve in the variation is the base curve $ \gamma = \gamma_{0}$. In particular, the traveltime for $ \gamma $ is $ t_0 $ and $ \tau $ is the time at which $ \gamma $ meets $ \eta $. However, the parametrization of any other curve $ \gamma_{\omega} $ is not arbitrary: it must satisfy the conditions $ \gamma_{\omega}(0) = q_1 $, $ \gamma_{\omega}(t_0) = q_2 $ and $ \gamma_{\omega}(\tau) \in \eta $. So, we cannot assume that $ \gamma_{\omega} $ is time-parametrized for $ \omega \neq 0 $, which means that the traveltime must be computed through \eqref{eq:traveltime} and $ \tau $ does not have to coincide with the actual time at which $ \gamma_{\omega} $ intersects $ \eta $ for $ \omega \neq 0 $.
\end{rem}

\begin{notation}
From now on we will work in coordinates with the following usual conventions. Let $ q = (q^1,\ldots,q^n) $ be a coordinate system on $ Q $ and consider also the coordinate system $ (q,\dot{q}) $ on $ TQ $, being $ \dot{q} = (\dot{q}^1,\ldots,\dot{q}^n) $ the natural coordinates on the fiber induced by $ q $. This means that every $ (p,v) \in TQ $ will be written is these coordinates as $ (p^1,\ldots,p^n,v^1,\ldots,v^n) \in \mathds{R}^{2n} $, where $ p^i = q^i(p) $ and $ v^i = \dot{q}^i(v) = v(q^i) $, i.e. $ v = v^i \frac{\partial}{\partial q^i}\rvert_p $ (using Einstein's summation convention). Then:
\begin{itemize}
\item Any function on $ TQ $ such as the Finsler metrics $ F_j $ can be written in coordinates as $ F_j \equiv F_j(q,\dot{q}) $ and it makes sense to compute the following derivatives for $ j = 1,2 $:
\begin{equation*}
\begin{split}
& \frac{\partial F_j}{\partial q}(p,v) := \left( \frac{\partial F_j}{\partial q^1}(p,v),\ldots,\frac{\partial F_j}{\partial q^n}(p,v) \right) \in \mathds{R}^n, \\
& \frac{\partial F_j}{\partial \dot{q}}(p,v) := \left( \frac{\partial F_j}{\partial \dot{q}^1}(p,v),\ldots,\frac{\partial F_j}{\partial \dot{q}^n}(p,v) \right) \in \mathds{R}^n,
\end{split}
\end{equation*}

\item A curve $ \gamma: [0,t_0] \rightarrow Q $ will be written in these coordinates as $ \gamma(t) \equiv q(t) = (q^1(t),\ldots,q^n(t)) $, where $ q^i(t) := q^i(\gamma(t)) $, and its tangent vector as $ \gamma'(t) \equiv \dot{q}(t) = (\dot{q}^1(t),\ldots,\dot{q}^n(t)) $, where
\begin{equation*}
\dot{q}^i(t) := \dot{q}^i(\gamma'(t)) = \frac{dq^i}{dt}(t), \qquad i = 1,\ldots,n.
\end{equation*}
This way, we can also write for $ j = 1,2 $:
\begin{equation*}
\begin{split}
& F_j(t) := F_j(q(t),\dot{q}(t)), \\
& \frac{\partial F_j}{\partial q}(t) := \frac{\partial F_j}{\partial q}(q(t),\dot{q}(t)), \\
& \frac{\partial F_j}{\partial \dot{q}}(t) := \frac{\partial F_j}{\partial \dot{q}}(q(t),\dot{q}(t)),
\end{split}
\end{equation*}

\item If we consider an admissible variation $ H $ of $ \gamma $, then we write in coordinates $ H(\omega,t) \equiv q_{\omega}(t) $, being $ \dot{q}_{\omega}(t) $ the tangent vector of the curve $ t \mapsto q_{\omega}(t) $. Also, the variational field is $ t \mapsto (q(t),W(t)) $, with $ W(t) = W^i(t) \frac{\partial}{\partial q^i}\rvert_{q(t)} \equiv (W^1(t),\ldots,W^n(t)) $ and
\begin{equation*}
W^i(t) = \left. \frac{\partial q^i_{\omega}(t)}{\partial \omega} \right\rvert_{\omega=0}, \qquad i=1,\ldots,n.
\end{equation*}
\end{itemize}

Finally, observe that even if $ q $ is not a global coordinate system on $ Q $, the compactness of $ [0,t_0] $ allows one to find a finite number of coordinate charts $ \lbrace(U_i,q_i)\rbrace_{i=1}^{\tilde{k}+\tilde{l}} $ and a partition of the interval $ 0 = r_1 < \ldots < r_{\tilde{k}} = \tau < \ldots < r_{\tilde{k}+\tilde{l}} = t_0 $ such that $ \gamma([r_i,r_{i+1}]) \subset U_i $. So, we can compute
\begin{equation*}
\begin{split}
T[\gamma] & = \int_{0}^{\tau} F_1(\gamma(t),\gamma'(t)) dt + \int_{\tau}^{t_0} F_2(\gamma(t),\gamma'(t)) dt = \\
& = \sum_{i=1}^{\tilde{k}-1} \int_{r_i}^{r_{i+1}} F_1(q_i(t),\dot{q}_i(t)) dt + \sum_{i=\tilde{k}}^{\tilde{k}+\tilde{l}-1} \int_{r_i}^{r_{i+1}} F_2(q_i(t),\dot{q}_i(t)) dt
\end{split}
\end{equation*}
and therefore, without loss of generality, we can assume that $ q $ is a global coordinate system on $ Q $ (and so is $ (q,\dot{q}) $ on $ TQ $).
\end{notation}

\subsection{Critical points of the traveltime functional}
\label{subsec:crit_points}
Let us now derive the conditions for the time-parametrized $ \gamma \in \mathcal{N} $ to be a critical point of $ T_{\mathcal{N}} $. If such a curve exists, then
\begin{equation}
\label{eq:integrals}
\begin{split}
0 = & \left. \frac{d}{d\omega} T_{\mathcal{N}}[\gamma_{\omega}] \right\rvert_{\omega=0} = \\
= & \left. \frac{d}{d\omega} \right\rvert_{\omega=0} \int_{0}^{\tau} F_1(q_{\omega}(t),\dot{q}_{\omega}(t))dt + \left. \frac{d}{d\omega} \right\rvert_{\omega=0} \int_{\tau}^{t_0} F_2(q_{\omega}(t),\dot{q}_{\omega}(t))dt
\end{split}
\end{equation}
for every admissible variation of $ \gamma $. The first integral is
\begin{equation*}
\begin{split}
I_1 = & \int_{0}^{\tau} \left. \frac{\partial}{\partial\omega} F_1(q_{\omega}(t),\dot{q}_{\omega}(t)) \right\rvert_{\omega=0} dt = \\
= & \int_{0}^{\tau} \left( \frac{\partial F_1}{\partial q^i}(t) W^i(t) + \frac{\partial F_1}{\partial \dot{q}^i}(t) \frac{d}{dt}W^i(t) \right) dt = \\
= & \int_{0}^{\tau} \left[ \left( \frac{\partial F_1}{\partial q^i}(t) - \frac{d}{dt}\frac{\partial F_1}{\partial \dot{q}^i}(t) \right) W^i(t) + \frac{d}{dt} \left(\frac{\partial F_1}{\partial \dot{q}^i}(t) W^i(t) \right) \right] dt = \\
= & \frac{\partial F_1}{\partial \dot{q}^i}(\tau^-) W^i(\tau) + \int_{0}^{\tau} \left( \frac{\partial F_1}{\partial q^i}(t) - \frac{d}{dt}\frac{\partial F_1}{\partial \dot{q}^i}(t) \right) W^i(t) dt.
\end{split}
\end{equation*}
Analogously, the second integral is
\begin{equation*}
I_2 = - \frac{\partial F_2}{\partial \dot{q}^i}(\tau^+) W^i(\tau) + \int_{\tau}^{t_0} \left( \frac{\partial F_2}{\partial q^i}(t) - \frac{d}{dt}\frac{\partial F_2}{\partial \dot{q}^i}(t) \right) W^i(t) dt.
\end{equation*}
Note that in both cases we have used that $ W(0) = W(t_0) = 0 $, since the variation has fixed endpoints. Then \eqref{eq:integrals} reduces to
\begin{equation*}
\begin{split}
0 = & \left[ \frac{\partial F_1}{\partial \dot{q}^i}(\tau^-) - \frac{\partial F_2}{\partial \dot{q}^i}(\tau^+) \right] W^i(\tau) + \int_{0}^{\tau} \left( \frac{\partial F_1}{\partial q^i}(t) - \frac{d}{dt}\frac{\partial F_1}{\partial \dot{q}^i}(t) \right) W^i(t) dt \\
& + \int_{\tau}^{t_0} \left( \frac{\partial F_2}{\partial q^i}(t) - \frac{d}{dt}\frac{\partial F_2}{\partial \dot{q}^i}(t) \right) W^i(t) dt
\end{split}
\end{equation*}
and this must hold for every admissible variation $ H(\omega,t) \equiv q_{\omega}(t) $, which implies that $ \gamma $ must satisfy
\begin{equation}
\label{eq:euler}
\begin{split}
& \frac{\partial F_1}{\partial q^i}(t) - \frac{d}{dt}\frac{\partial F_1}{\partial \dot{q}^i}(t) = 0, \qquad \forall t \in [0,\tau), \\
& \frac{\partial F_2}{\partial q^i}(t) - \frac{d}{dt}\frac{\partial F_2}{\partial \dot{q}^i}(t) = 0, \qquad \forall t \in (\tau,t_0],
\end{split}
\end{equation}
and the following additional condition at the break point $ \gamma(\tau) $:
\begin{equation}
\label{eq:break}
\left[ \frac{\partial F_1}{\partial \dot{q}^i}(\tau^-) - \frac{\partial F_2}{\partial \dot{q}^i}(\tau^+) \right] W^i(\tau) = 0, \qquad \forall H.
\end{equation}

On the one hand, the two equations in \eqref{eq:euler} are the well-known Euler-Lagrange equations, equivalent to the (pre-)geodesic equations for the metrics $ F_1, F_2 $. Along with the time-parametrization conditions in \eqref{eq:time_par}, this tells us that $ \gamma $ is an $ F_1 $-unit speed geodesic from $ \gamma(0) = q_1 \in Q_1 $ to $ \gamma(\tau) \in \eta $ and an $ F_2 $-unit speed geodesic from there to $ \gamma(t_0) = q_2 \in Q_2 $. On the other hand, \eqref{eq:break} is the condition at the break point $ \gamma(\tau) $ that provides the desired Snell's law, as we will see next.

\subsection{Generalized Snell's law}
Let us now examine the equation \eqref{eq:break}. First, note that for every admissible variation $ H(\omega,t) \equiv q_{\omega}(t) $, we have that $ q_{\omega}(\tau) \in \eta $ for all $ \omega \in (-\varepsilon,\varepsilon) $, where $ \tau $ is fixed (recall Definition~\ref{def:variation}). Therefore
\begin{equation*}
\left. \frac{\partial q_{\omega}(\tau)}{\partial\omega} \right\rvert_{\omega=0} = W(\tau) \in T_{q(\tau)}\eta
\end{equation*}
and this restriction must hold for all $ H $, so \eqref{eq:break} can be rewritten as
\begin{equation}
\label{eq:snell}
\left[ \frac{\partial F_1}{\partial \dot{q}^i}(\tau^-) - \frac{\partial F_2}{\partial \dot{q}^i}(\tau^+) \right] u^i = 0, \qquad \forall u \in T_{q(\tau)}\eta,
\end{equation}
where $ u^i = \dot{q}^i(u) $. This is the condition at the break point that governs the change in the trajectory's direction, and thus we arrive at the general version of Snell's law:
\begin{thm}
\label{th:refraction}
Let $ \gamma: [0,t_0] \rightarrow Q $ be a time-parametrized curve in $ \mathcal{N} $. Then $ \gamma $ is a critical point of $ T_{\mathcal{N}} $ if and only if it is an $ F_1 $-unit speed geodesic from $ \gamma(0) \in Q_1 $ to $ \gamma(\tau) \in \eta $, an $ F_2 $-unit speed geodesic from $ \gamma(\tau) \in \eta $ to $ \gamma(t_0) \in Q_2 $, and it satisfies the generalized Snell's law given by \eqref{eq:snell} at the break point $ \gamma(\tau) \in \eta $. In this case we say that $ \gamma $ is a {\em refracted trajectory}.
\end{thm}

\begin{rem}
The formulation we have carried out here is typical in classical mechanics, where instead of a Finsler metric $ F $ one has a Lagrangian defined as $ L = K-P $, being $ K $ and $ P $ the total kinetic and potential energy of a system of particles, respectively, and instead of minimizing the traveltime functional $ T = \int F $ one looks for the critical points of the action functional $ S = \int L $. In this setting, \eqref{eq:snell} states that the total linear momentum of the system must be conserved in the tangent directions to $ \eta $ (see e.g. \cite{FMOW}).
\end{rem}

Snell's law is essentially an equation system involving $ q(\tau) $, $ \dot{q}(\tau^-) $ and $ \dot{q}(\tau^+) $. Typically, $ q(\tau) $, $ \dot{q}(\tau^-) $ are known from the initial conditions and the geodesic equations in the first region $ Q_1 $, so one looks for $ \dot{q}(\tau^+) $. Since $ \text{dim}(\eta)=\text{dim}(T_{q(\tau)}\eta) = n-1 $, \eqref{eq:snell} provides exactly $ n-1 $ independent equations. With the additional condition $ F_2(q(\tau),\dot{q}(\tau^+))=1 $, we have $ n $ equations to find the $ n $-component vector $ \dot{q}(\tau^+) $. Alternatively, one can just find the \emph{direction} of $ \dot{q}(\tau^+) $ from \eqref{eq:snell}, which is usually more convenient in dimension $ 2 $, as we will see next.

\subsection{The two-dimensional case}
\label{subsec:two-dim}
If $ \text{dim}(Q) = n = 2 $, then the direction of the incident and refracted trajectory can be written in terms of the following angles.

\begin{defi}
Let $ n=2 $ and $ \gamma \in \mathcal{N} $. We define the {\em angles of incidence} and {\em refraction} $ \theta_1, \theta_2 \in (-\pi,\pi] $ as the (oriented) directions of $ \dot{q}(\tau^-), \dot{q}(\tau^+) $, respectively, measured with respect to the $ \dot{q}^1 $-axis. Namely,
\begin{equation*}
\tan\theta_1 = \frac{\dot{q}^2(\tau^-)}{\dot{q}^1(\tau^-)}, \qquad \tan\theta_2 = \frac{\dot{q}^2(\tau^+)}{\dot{q}^1(\tau^+)},
\end{equation*}
with the restriction that $ v_{\theta_1} := (\cos\theta_1,\sin\theta_1) $ and $ v_{\theta_2} := (\cos\theta_2,\sin\theta_2) $ must point to the $ Q_2 $-side of $ T_{q(\tau)}\eta $.
\end{defi}

This way, for a fixed break point $ q(\tau) \in \eta $, Snell's law is transformed into an equation between the angles of incidence and refraction. Specifically, it can always be written explicitly in terms of $ \tan\theta_1 $ and $ \tan\theta_2 $. Indeed, since $ F_1 $ and $ F_2 $ are positively homogeneous of degree one, $ F_j(\lambda v) = |\lambda| F_j(\text{sgn}(\lambda)v) $ for all $ v \in T_{q(\tau)}\eta $, $ \lambda \in \mathds{R} $ and, as a consequence,
\begin{equation*}
\frac{\partial F_j}{\partial \dot{q}}(\lambda v) = \frac{\partial F_j}{\partial \dot{q}}(\text{sgn}(\lambda)v), \qquad j=1,2,
\end{equation*}
so, in particular,
\begin{equation*}
\begin{split}
\frac{\partial F_j}{\partial \dot{q}^i}(\dot{q}^1,\dot{q}^2) = & \frac{\partial F_j}{\partial \dot{q}^i}\left( \text{sgn}(\dot{q}^1) \left( 1,\frac{\dot{q}^2}{\dot{q}^1} \right) \right) = \frac{\partial F_j}{\partial \dot{q}^i}(\text{sgn}(\dot{q}^1)(1,\tan\theta_j)),
\end{split}
\end{equation*}
for $ i,j=1,2 $, where the references to $ q(\tau) $ and $ \tau^\pm $ have been omitted.

Therefore, in general one can solve \eqref{eq:snell} for $ \tan\theta_2 $ and then obtain $ \theta_2 \in (-\pi,\pi] $ by imposing the restriction that $ v_{\theta_2} $ must point to the $ Q_2 $-side of $ T_{q(\tau)}\eta $.

\begin{exe}
\label{ex:classical_snell}
Let us recover the classical Snell's law in the Euclidean plane. In this usual setting, $ Q = \mathds{R}^2 $, $ (q,\dot{q}) = (x,y,\dot{x},\dot{y}) $ and the background Riemannian metric $ g = dx^2+dy^2 $ is the natural one in $ \mathds{R}^2 $. Let $ Q_1 = \lbrace (x,y) \in \mathds{R}^2: x<0 \rbrace $ and $ Q_2 = \lbrace (x,y) \in \mathds{R}^2: x>0 \rbrace $, so $ \eta $ is the vertical line $ x = 0 $. In this situation, assume that the propagation of light is isotropic and its velocity depends on the refractive index of the medium. Assume also that $ Q_1 $ and $ Q_2 $ are homogeneous, i.e. their respective refractive indices $ n_1, n_2 $ are constant. Then
\begin{equation*}
\begin{split}
& V_1(x,y,\dot{x},\dot{y}) = \frac{c}{n_1} \quad \Rightarrow \quad F_1(x,y,\dot{x},\dot{y}) = \frac{n_1}{c}\sqrt{\dot{x}^2+\dot{y}^2}, \\
& V_2(x,y,\dot{x},\dot{y}) = \frac{c}{n_2} \quad \Rightarrow \quad F_2(x,y,\dot{x},\dot{y}) = \frac{n_2}{c}\sqrt{\dot{x}^2+\dot{y}^2},
\end{split}
\end{equation*}
where $ c $ is the speed of light in vacuum. The angles of incidence and refraction are given by
\begin{equation*}
\tan\theta_1 = \frac{\dot{y}(\tau^-)}{\dot{x}(\tau^-)}, \qquad \tan\theta_2 = \frac{\dot{y}(\tau^+)}{\dot{x}(\tau^+)},
\end{equation*}
so, omitting the reference to $ \tau^\pm $ we get
\begin{equation*}
\begin{split}
\frac{\partial F_j}{\partial \dot{y}} = & \frac{n_j}{c} \frac{\dot{y}}{\sqrt{\dot{x}^2+\dot{y}^2}} = \frac{n_j}{c} \frac{\dot{y}}{\lvert \dot{x} \rvert \sqrt{1+\dot{y}^2/\dot{x}^2}} = \pm \frac{n_j}{c} \frac{\tan\theta_j}{\sqrt{1+\tan^2\theta_j}} = \\
= & \pm \frac{n_j}{c}\tan\theta_j |\cos\theta_j| = \frac{n_j}{c} \sin\theta_j, \qquad j=1,2,
\end{split}
\end{equation*}
where the $ \pm $ vanishes because $ \text{sgn}(\dot{x}(\tau^-)) = \text{sgn}(\cos\theta_1) $ and $ \text{sgn}(\dot{x}(\tau^+)) = \text{sgn}(\cos\theta_2) $. This way, since in this particular example $ \{(0,1)\} $ is a basis of $ T_{(0,y(\tau))}\eta $ at any break point $ (0,y(\tau)) \in \eta $, \eqref{eq:snell} yields the classical Snell's law:
\begin{equation}
\label{eq:snell_classic}
n_1\sin\theta_1 = n_2\sin\theta_2,
\end{equation}
with the restriction that $ \text{sgn}(\cos\theta_1) = \text{sgn}(\cos\theta_2) > 0 $, so that both $ v_{\theta_1} $ and $ v_{\theta_2} $ point to $ Q_2 $.
\end{exe}

\section{Reflected trajectories}
\label{sec:reflection}
\subsection{The law of reflection}
So far, given a curve $ \gamma $ departing from $ Q_1 $ and intersecting $ \eta $, we have been looking for the refracted trajectory, i.e. the one that crosses $ \eta $ and continues into $ Q_2 $. It is interesting also to seek the reflected trajectory, which will be the one that minimizes the traveltime in a slightly different set of curves than $ \mathcal{N} $ in Definition~\ref{def:traveltime}.

Specifically, fixing $ [0,t_0] $, $ \tau \in (0,t_0) $ and two points $ q_1, q_2 \in Q_1 $, we define $ \mathcal{N}^* $ as the set of all (regular) piecewise smooth curves $ \gamma: [0,t_0] \rightarrow Q $ from $ \gamma(0) = q_1 $ to $ \gamma(t_0) = q_2 $ such that $ \gamma $ remains entirely in $ \overline{Q}_1 $, touching $ \eta $ once at $ \gamma(\tau) $.

If one now looks for the critical points of $ T_{\mathcal{N}^*}: \mathcal{N}^* \rightarrow \mathds{R} $ in a similar way as in the previous section, one finds, on the one hand, the Euler-Lagrange equations (compare with \eqref{eq:euler})
\begin{equation*}
\frac{\partial F_1}{\partial q^i}(t) - \frac{d}{dt}\frac{\partial F_1}{\partial q^i}(t) = 0, \qquad \forall t \in [0,\tau) \cup (\tau,t_0],
\end{equation*}
which now state that the candidate curve $ \gamma(t) \equiv q(t) $ must be an $ F_1 $-geodesic from $ q_1 $ to $ \eta $ and also from $ \eta $ to $ q_2 $, and on the other hand, instead of Snell's law \eqref{eq:snell} (i.e. the refraction law), we get  the following reflection law:
\begin{equation}
\label{eq:reflection}
\left[ \frac{\partial F_1}{\partial \dot{q}^i}(\tau^-) - \frac{\partial F_1}{\partial \dot{q}^i}(\tau^+) \right] u^i = 0, \qquad \forall u \in T_{q(\tau)}\eta.
\end{equation}

Therefore, analogously to Theorem~\ref{th:refraction}, we have the following result.
\begin{thm}
\label{th:reflection}
Let $ \gamma: [0,t_0] \rightarrow Q $ be a time-parametrized curve in $ \mathcal{N}^* $. Then $ \gamma $ is a critical point of $ T_{\mathcal{N}^*} $ if and only if it is an $ F_1 $-unit speed geodesic from $ \gamma(0) \in Q_1 $ to $ \gamma(\tau) \in \eta $ and from $ \gamma(\tau) \in \eta $ to $ \gamma(t_0) \in Q_1 $, and it satisfies the law of reflection given by \eqref{eq:reflection} at the break point $ \gamma(\tau) \in \eta $. In this case we say that $ \gamma $ is a {\em reflected trajectory}.
\end{thm}

Moreover, the discussion developed in \S~\ref{subsec:two-dim} becomes valid here as well, i.e. in the two-dimensional case the law of reflection can be written in terms of the following angles.

\begin{defi}
\label{def:angles}
Let $ n = 2 $ and $ \gamma \in \mathcal{N}^* $. We define the {\em angles of incidence} and {\em reflection} $ \theta_1, \theta_3 \in (-\pi,\pi] $ as the (oriented) directions of $ \dot{q}(\tau^-), \dot{q}(\tau^+) $, respectively, measured with respect to the $ \dot{q}^1 $-axis. Namely,
\begin{equation*}
\tan\theta_1 = \frac{\dot{q}^2(\tau^-)}{\dot{q}^1(\tau^-)}, \qquad \tan\theta_3 = \frac{\dot{q}^2(\tau^+)}{\dot{q}^1(\tau^+)},
\end{equation*}
with the restriction that $ v_{\theta_1} := (\cos\theta_1,\sin\theta_1) $ must point to the $ Q_2 $-side and $ v_{\theta_3} := (\cos\theta_3,\sin\theta_3) $ to the $ Q_1 $-side of $ T_{q(\tau)}\eta $.
\end{defi}

Note the difference between the angles of refraction and reflection (compare with Definition~\ref{def:angles}): now $ \dot{q}(\tau^+) $ must point to the $ Q_1 $-side of $ T_{q(\tau)}\eta $ because otherwise, $ \gamma $ would cross to $ Q_2 $ and thus, $ \gamma \notin \mathcal{N}^* $ (see Figure~\ref{fig:plane}).

\begin{figure}
\centering
\includegraphics[width=1\textwidth]{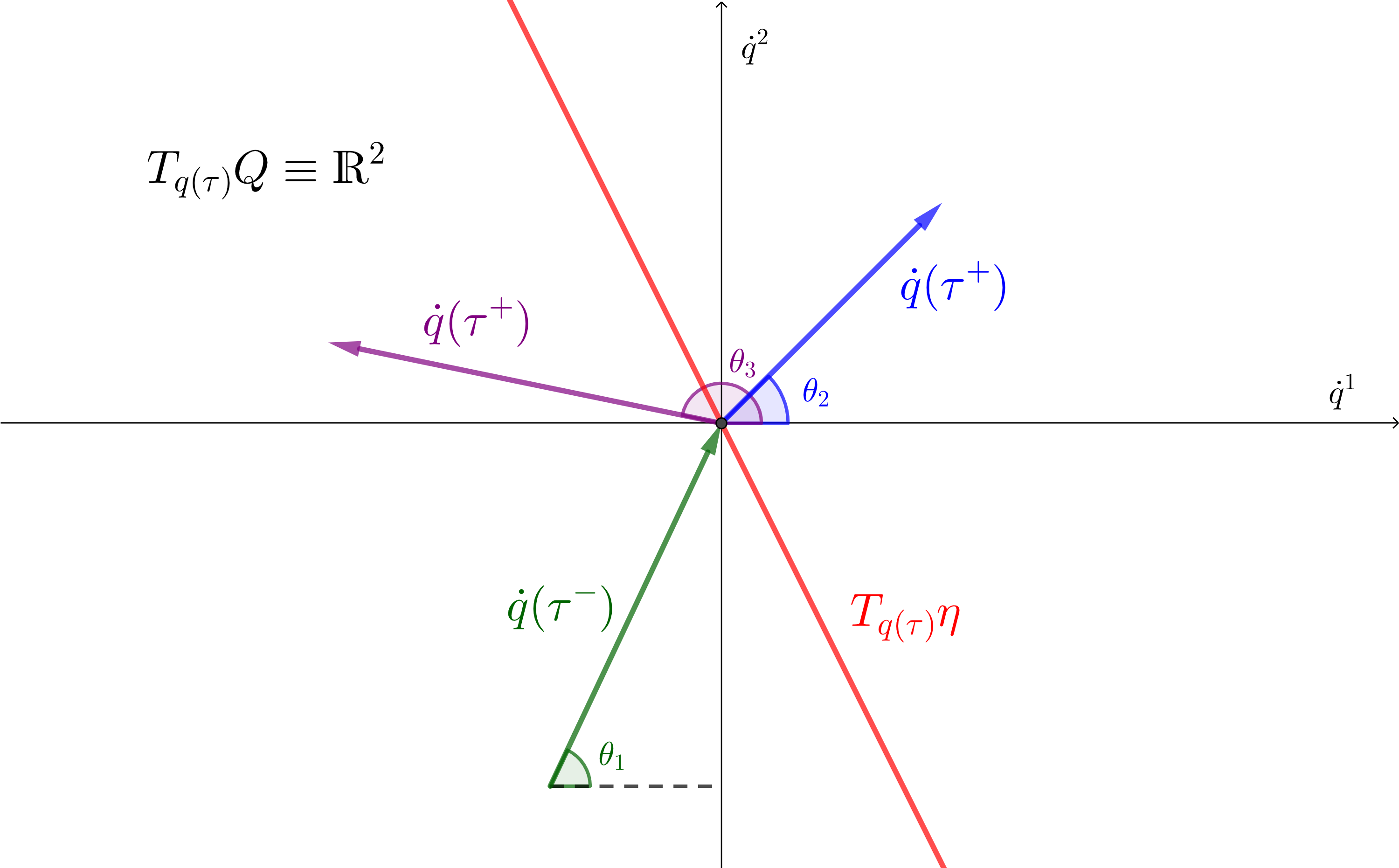}
\caption{In the two-dimensional case, Snell's law or the law of reflection can be written in terms of the angle of incidence $ \theta_1 $ and the angle of refraction $ \theta_2 $ or reflection $ \theta_3 $, respectively. When looking for the refracted trajectory $ \gamma \in \mathcal{N} $ (resp. the reflected trajectory $ \gamma \in \mathcal{N}^* $), $ \dot{q}(\tau^+) $ must point to the $ Q_2 $-side (resp. the $ Q_1 $-side) of $ T_{q(\tau)}\eta $.}
\label{fig:plane}
\end{figure}

\begin{exe}
\label{ex:classical_reflection}
In the same setting as in Example~\ref{ex:classical_snell}, let us recover the classical reflection law. The angles of incidence and reflection are
\begin{equation*}
\tan\theta_1 = \frac{\dot{y}(\tau^-)}{\dot{x}(\tau^-)}, \qquad \tan\theta_3 = \frac{\dot{y}(\tau^+)}{\dot{x}(\tau^+)},
\end{equation*}
with $ \text{sgn}(\dot{x}(\tau^-)) = \text{sgn}(\cos\theta_1) $ and $ \text{sgn}(\dot{x}(\tau^+)) = \text{sgn}(\cos\theta_2) $, so
\begin{equation*}
\frac{\partial F_1}{\partial \dot{y}}(\tau^-) =  \frac{n_1}{c} \sin\theta_1, \qquad \frac{\partial F_1}{\partial \dot{y}}(\tau^+) =  \frac{n_1}{c} \sin\theta_3.
\end{equation*}
Thus, \eqref{eq:reflection} directly yields
\begin{equation}
\label{eq:reflection_classical}
\sin\theta_1 = \sin\theta_3,
\end{equation}
with the restriction that $ \text{sgn}(\cos\theta_1) > 0 $ and $ \text{sgn}(\cos\theta_3) < 0 $, so that $ v_{\theta_1} $ points to $ Q_2 $ and $ v_{\theta_3} $ to $ Q_1 $. For example, if $ \theta_1 \in [0,\frac{\pi}{2}) $, then $ \theta_3 = \pi-\theta_1 $. If $ \theta_1 \in (-\frac{\pi}{2},0) $, then $ \theta_3 = -\pi-\theta_1 $.\footnote{Another usual way to define the angles of incidence, refraction and reflection is to restrict their values to $ [0,\frac{\pi}{2}) $ and measure them with respect to the normal direction to $ T_{(x(\tau),y(\tau))}\eta $. In this case, the classical Snell's law would still be \eqref{eq:snell_classic} but without any additional restriction, and the classical reflection law would read $ \theta_1 = \theta_3 $.}
\end{exe}

\subsection{The total reflection phenomenon}
In general, the existence of solutions to Snell's law \eqref{eq:snell} is not guaranteed. Namely, given an incident velocity $ \dot{q}(\tau^-) $, there might not be a solution $ \dot{q}(\tau^+) $ pointing to the $ Q_2 $-side of $ T_{q(\tau)}\eta $. In this case, there is no refraction and the incoming trajectory is totally reflected without being able to cross into $ Q_2 $. This is known as the {\em total reflection phenomenon}. The limit between the existence and non-existence of solution is given by the situation where $ \dot{q}(\tau^+) $ is tangent to $ T_{q(\tau)}\eta $, which motivates the following definition.

\begin{defi}
We say that an incident velocity $ \dot{q}(\tau^-) $ (or an angle of incidence $ \theta_1 $ in the two-dimensional case) is {\em critical} if there exists a solution $ \dot{q}(\tau^+) $ to Snell's law \eqref{eq:snell} which is tangent to $ T_{q(\tau)}\eta $.
\end{defi}

The trajectory $ \gamma $ associated with the critical velocity can still be considered as an actual refracted trajectory as long as immediately after touching $ \eta $, it passes to $ Q_2 $ or remains on $ \eta $ with the $ F_{2} $ metric, i.e. if there exists $ \delta > 0 $ such that $ \gamma(\tau + t) \in (\overline{Q}_2, F_{2}) $ for all $ t \in (0,\delta) $. We say in this case that $ \gamma $ is a {\em critical trajectory}. In fact, critical trajectories along $ \eta $ are relevant in situations such as the one we will describe in the next section.

\begin{exe}
\label{exe:Classic2}
In the classical setting of Example~\ref{ex:classical_snell}, $ v_{\theta_2} $ is tangent to $ \eta $ when $ \theta_2 = \pm \frac{\pi}{2} $, so the critical angles are given by
\begin{equation*}
\theta^{\pm}_c = \pm \arcsin\left(\frac{n_2}{n_1}\right).
\end{equation*}
Note that these angles only exist if $ n_1 > n_2 $, i.e. if the velocity $ V_{2} $ on $ Q_2 $ is greater than $ V_{1} $ on $ Q_1 $. If, for example, $ n_1 = 1 $ and $ n_2 = \frac{1}{2} $, then $ \theta^{\pm}_c = \pm \frac{\pi}{6} $. In the cases $ \theta_1 > \frac{\pi}{6} $ or $ \theta_1 < -\frac{\pi}{6} $, there is no angle $ \theta_2 $ satisfying the classical Snell's law \eqref{eq:snell_classic} and pointing to $ Q_2 $. In consequence, there does not exist a time-minimizing path from any  $ q_{1} \in Q_{1} $ to any $ q_{2} \in Q_{2} $ which meets $ \eta $ with $ \theta_1 > \frac{\pi}{6} $ or $ \theta_1 < -\frac{\pi}{6} $, so there is no refraction and the incident trajectory is totally reflected with an angle $ \theta_3 $ satisfying the classical law of reflection \eqref{eq:reflection_classical}. In other terms, suppose $ q_{1} $ has Euclidean distance $ d $ to $ \eta $, then the Euclidean size of the ``window'' through which a time-minimizing trajectory can go through $ \eta $ is
$ \frac{2d\cdot n_{2}}{\sqrt{n_{1}^2 - n_{2}^{2}}} $.
\end{exe}

\section{Global considerations in $ \mathds{R}^2 $}
\label{sec:R2}
In order to reduce the casuistry we can find in the most general case and fix some ideas, we will reduce our study in this section to the classical setting (see Example~\ref{ex:classical_snell}), although maintaining the Finslerian nature of the metrics. Namely, $ Q = \mathds{R}^2 $, $ (q,\dot{q}) = (x,y,\dot{x},\dot{y}) $, $ g = dx^2+dy^2 $, $ Q_1 = \lbrace (x,y) \in \mathds{R}^2: x<0 \rbrace $, $ Q_2 = \lbrace (x,y) \in \mathds{R}^2: x>0 \rbrace $ and $ F_1, F_2 $ are Finsler metrics independent of the position (i.e. Minkowski norms):
\begin{equation*}
F_j(x,y,\dot{x},\dot{y}) = F_j(\dot{x},\dot{y}) = \frac{\sqrt{\dot{x}^2+\dot{y}^2}}{V_j(\dot{x},\dot{y})}, \qquad j=1,2.
\end{equation*}
In particular, this means that the geodesics of $ F_1, F_2 $ are straight lines. Note also that the angles of incidence, refraction and reflection are restricted to the values $ \theta_1 \in (-\frac{\pi}{2},\frac{\pi}{2}) $, $ \theta_2 \in [-\frac{\pi}{2},\frac{\pi}{2}] $, $ \theta_3 \in (-\pi,-\frac{\pi}{2}) \cup (\frac{\pi}{2},\pi] $, and there are at most two critical angles $ \theta_c^\pm $, with $ \theta_c^+ > \theta_c^- $: $ \theta_c^+ $ corresponds to $ \theta_2 = \frac{\pi}{2} $ and $ \theta_c^- $ to $ \theta_2 = -\frac{\pi}{2} $ (the uniqueness of $ \theta_c^\pm $ when they exist is proven below in Lemma~\ref{lem:P}).

\subsection{Existence and uniqueness of solution}
In this setting, Snell's law \eqref{eq:snell} reduces to
\begin{equation*}
\frac{\partial F_1}{\partial\dot{y}}(\tau^-) = \frac{\partial F_2}{\partial\dot{y}}(\tau^+).
\end{equation*}
Now, being $ F_1 $ and $ F_2 $ Minkowski norms, we have that $ \frac{\partial F_j}{\partial y} = 0 $ and thus, from the Euler-Lagrange equations \eqref{eq:euler}:
\begin{equation*}
\begin{split}
& \frac{d}{dt}\frac{\partial F_1}{\partial \dot{y}}(t) = 0, \qquad \forall t \in [0,\tau), \\
& \frac{d}{dt}\frac{\partial F_2}{\partial \dot{y}}(t) = 0, \qquad \forall t \in (\tau,t_0].
\end{split}
\end{equation*}
which means, together with Snell's law, that there is a conserved quantity along the entire refracted curve $ \gamma \in \mathcal{N} $.\footnote{This is related to Noether's theorem: the translational symmetry of the setting along the $ y $-axis implies the existence of a {\em constant of motion}.} This constant is called the {\em raypath parameter} (see \cite{BS}).  It  only depends on the initial direction, due to the zero-homogeneity of $ \frac{\partial F_j}{\partial \dot{y}} $. So, Snell's law becomes
\begin{equation}
\label{eq:snell_P}
P_1(\theta_1) = P_2(\theta_2),
\end{equation}
with
\begin{equation}
\label{eq:raypath}
P_j(\theta) := \frac{\partial F_j}{\partial\dot{y}}(\theta) = \frac{\sin\theta}{V_j(\theta)} + \cos\theta\frac{\partial}{\partial \theta} \left(\frac{1}{V_j(\theta)}\right), \qquad j=1,2,
\end{equation}
where we have used that $ \frac{\partial}{\partial \dot{y}} = \frac{\cos^2\theta}{\dot{x}}\frac{\partial}{\partial\theta} $. Note that $ P_j(\theta) $ is smooth, thanks to the smoothness and positiveness of $ V_j(\theta) $.

\begin{lemma}
\label{lem:P}
The functions $ P_j: [-\frac{\pi}{2},\frac{\pi}{2}] \rightarrow \mathds{R} $ are strictly increasing. Also, the critical angle $ \theta_c^+ $ (resp. $ \theta_c^- $) exists and is necessarily unique if and only if $ V_1(\frac{\pi}{2}) < V_2(\frac{\pi}{2}) $ (resp. $ V_1(-\frac{\pi}{2}) < V_2(-\frac{\pi}{2}) $).\footnote{As in the classical case, for both critical angles $ \theta_c^\pm $ to exist we need the velocity $ V_2 $ to be greater than $ V_1 $ on $ \eta $ in both directions $ \pm \frac{\pi}{2} $.}
\end{lemma}
\begin{proof}
One can easily verify that
\begin{equation*}
\frac{d P_j}{d\theta} = \frac{\cos\theta}{V_j^3} \left[ V_j^2 + 2\left(\frac{d V_j}{d\theta}\right)^2-V_j\frac{d^2 V_j}{d\theta^2} \right],
\end{equation*}
where the term in brackets is exactly the second fundamental form of the indicatrix $ \Sigma_j $ of $ F_j $ (given by the curve $ (\pi,\pi] \ni \theta \mapsto V_j(\cos\theta,\sin\theta) $) with respect to the natural Euclidean metric in $ \mathds{R}^2 $ and the inner normal vector. The strong convexity of $ \Sigma_j $ and the positiveness of $ V_j $ ensure that $ \frac{d P_j}{d\theta} > 0 $ for all $ \theta \in (-\frac{\pi}{2},\frac{\pi}{2}) $, so we conclude that $ P_j(\theta) $ is strictly increasing in $ \theta \in [-\frac{\pi}{2},\frac{\pi}{2}] $. In particular, this means that the maximum of $ P_j(\theta) $ is $ P_j(\frac{\pi}{2}) = \frac{1}{V_j(\frac{\pi}{2})} $ and the minimum $ P_j(-\frac{\pi}{2}) = -\frac{1}{V_j(-\frac{\pi}{2})} $. 
The critical angle $ \theta_c^+ $ (resp. $ \theta_c^- $) is the angle of incidence satisfying $ P_1(\theta_c^+) = P_2(\frac{\pi}{2}) $ (resp. $ P_1(\theta_c^-) = P_2(-\frac{\pi}{2}) $), so it exists if and only if $ P_1(-\frac{\pi}{2}) < P_2(\frac{\pi}{2}) < P_1(\frac{\pi}{2}) $ (resp. $ P_1(-\frac{\pi}{2}) < P_2(-\frac{\pi}{2}) < P_1(\frac{\pi}{2}) $), but note that the first (resp. second) inequality is guaranteed, as $ P_j(\frac{\pi}{2}) > 0 $ and $ P_j(-\frac{\pi}{2}) < 0 $ for all $ j=1,2 $. Also, the uniqueness of $ \theta_c^\pm $ follows from the fact that $ P_1 $ is strictly increasing.
\end{proof}

\begin{thm}
\label{th:existence_refraction}
Given any angle of incidence $ \theta_1 \in (\theta_c^-,\theta_c^+) $ (replace $ \theta_c^\pm $ with $ \pm \frac{\pi}{2} $ if the critical angles do not exist), there exists a unique angle of refraction $ \theta_2 \in (-\frac{\pi}{2},\frac{\pi}{2}) $ satisfying Snell's law. Equivalently, given any $ q_1 \in Q_1 $ and $ q_2 \in Q_2 $, there exists a unique refracted trajectory $ \gamma \in \mathcal{N} $ joining both points.
\end{thm}
\begin{proof}
Given $ \theta_1 \in (\theta_c^-,\theta_c^+) $ (or $ \theta_1 \in (-\frac{\pi}{2},\frac{\pi}{2}) $ if the critical angles do not exist), from Lemma~\ref{lem:P} we know that $ P_2(-\frac{\pi}{2}) < P_1(\theta_1) < P_2(\frac{\pi}{2}) $, so there exists $ \theta_2 \in (-\frac{\pi}{2},\frac{\pi}{2}) $ satisfying \eqref{eq:snell_P} and since $ P_2(\theta) $ is strictly increasing, this angle is necessarily unique. Note that the converse is also true, i.e. given any $ \theta_2 \in (-\frac{\pi}{2},\frac{\pi}{2}) $ we can find $ \theta_1 \in (\theta_c^-,\theta_c^+) $ (or $ \theta_1 \in (-\frac{\pi}{2},\frac{\pi}{2}) $) satisfying Snell's law. In $ \mathds{R}^2 $ with straight lines as geodesics, this ensures that all the refracted trajectories cover $ Q_2 $ entirely, independently of the initial point $ q_1 \in Q_1 $.
\end{proof}

Of course, one can follow the same steps and derive an analogous result for the existence and uniqueness of solutions to the law of reflection, which reads $ P_1(\theta_1) = P_1(\theta_3) $ in terms of the raypath parameter.

\begin{thm}
\label{th:existence_reflection}
Given any angle of incidence $ \theta_1 \in (-\frac{\pi}{2},\frac{\pi}{2}) $, there exists a unique angle of reflection $ \theta_3 \in (-\pi,-\frac{\pi}{2}) \cup (\frac{\pi}{2},\pi] $ satisfying the law of reflection. Equivalently, given any $ q_1, q_2 \in Q_1 $, there exists a unique refracted trajectory $ \gamma \in \mathcal{N}^* $ joining both points.
\end{thm}

\subsection{New reflected trajectories}
\label{subsec:new_reflection}
The existence and uniqueness of refracted and reflected trajectories ensure that, once the initial and final points are fixed, there is only one critical point of $ T_{\mathcal{N}} $ and $ T_{\mathcal{N}^*} $, which therefore must be a global minimum. However, recall that curves in $ \mathcal{N} $ and $ \mathcal{N}^* $ have the restriction that they must intersect $ \eta $ only once. One might wonder if more contact points with $ \eta $ could result in a curve that spends even less time than the usual refracted or reflected trajectory. Namely, we can define new sets $ \hat{\mathcal{N}} $ and $ \hat{\mathcal{N}}^* $ in the same way as $ \mathcal{N} $ and $ \mathcal{N}^* $, respectively, but with the restriction that the (piecewise smooth) curves must meet $ \eta $ at least once (but possibly at more points), and look for the critical points of the traveltime functional on these sets.

Let $ \gamma $ be a curve from $ q_1 $ to $ q_2 $ that meets $ \eta $ for the first time at $ \tau_1 $ and leaves it for the last time at $ \tau_2 $. In order to minimize the traveltime, the curve from $ \gamma(\tau_1) $ to $ \gamma(\tau_2) $ must be the segment of the straight line $ \eta $ joining these points (since it is a geodesic in the present $ \mathds{R}^{2} $-setting under consideration). Now, if $ \gamma \in \hat{\mathcal{N}} $, by the triangle inequality (see e.g. \cite[Thm.~1.2.2]{BCS}) any two-segment trajectory from $ q_1 \in Q_1 $ to $ \gamma(\tau) \in \eta $ (for any $ \tau \in [\tau_1,\tau_2] $) and from there to $ q_2 \in Q_2 $ spends less time than $ \gamma $. So the critical points of $ T_{\hat{\mathcal{N}}} $ must always belong to $ \mathcal{N} $, i.e. they are the usual two-segment refracted trajectories crossing $ \eta $ once. However, if $ \gamma \in \hat{\mathcal{N}}^* $, the travel along $ \eta $ with the metric $ F_{2} $ can be an advantage. A real-world example is when moving between two points in the sea next to the shore: it can be faster to swim first to the shore, then run along the beach line edge and finally swim again to the final point. Obviously, the traveltime on $ \eta $ has to be computed through $ F_{2} $ (otherwise, the usual reflected trajectory and hence also the direct straight line trajectory in $ Q_{1} $ is always faster):
\begin{equation*}
T_{\hat{\mathcal{N}}^*}[\gamma] = \int_{0}^{\tau_1} F_1(\gamma'(t)) dt + \int_{\tau_1}^{\tau_2} F_2(\gamma'(t)) dt + \int_{\tau_2}^{t_0} F_1(\gamma'(t)) dt.
\end{equation*}
Taking this into account, we can now define a suitable variation and apply the calculus of variations in the same way as in \S~\ref{subsec:crit_points}. This gives the following conditions for a time-parametrized curve $ \gamma $ to be a critical point of $ T_{\hat{\mathcal{N}}^*} $:
\begin{itemize}
\item $ \gamma $ is composed of three straight lines: an $ F_1 $-unit segment $ \gamma_1 $ from $ q_1 \in Q_1 $ to $ \gamma(\tau_1) \in \eta $, an $ F_2 $-unit segment $ \gamma_2 \subset \eta $ between $ \gamma(\tau_1) $ and $ \gamma(\tau_2) $, and a final $ F_1 $-unit segment $ \gamma_3 $ from $ \gamma(\tau_2) \in \eta $ to $ q_2 \in Q_1 $ (see Figure~\ref{fig:new_reflection}).

\item At the break points, $ \gamma(t) \equiv (x(t),y(t)) $ (for $ t = \tau_{1} $ and $ t = \tau_{2} $, respectively)  must satisfy
\begin{equation*}
\begin{split}
& \frac{\partial F_1}{\partial\dot{y}}(\tau_1^-) = \frac{\partial F_2}{\partial\dot{y}}(\tau_1^+), \\
& \frac{\partial F_2}{\partial\dot{y}}(\tau_2^-) = \frac{\partial F_1}{\partial\dot{y}}(\tau_2^+).
\end{split}
\end{equation*}
Equivalently, in terms of the (constant) directions $ \theta_1,\theta_2,\theta_3 \in (-\pi,\pi] $ of each segment and the raypath parameter \eqref{eq:raypath}:
\begin{equation}
\label{eq:ref_R2}
P_1(\theta_1) = P_2(\theta_2) = P_1(\theta_3).
\end{equation}
In particular, this means that $ \theta_1 $ must be critical (so that $ \theta_2 = \pm \frac{\pi}{2} $) and $ \theta_3 $ must be exactly the angle of reflection we would obtain at $ \tau_1 $, being $ \gamma_1 $ the incident trajectory (see Figure~\ref{fig:new_reflection}).
\end{itemize}

If it exists, this three-segment trajectory is always faster than the usual reflection. However, its existence is not guaranteed: it depends on whether the critical angle exists (which in turn depends on whether the velocity $ V_{2} $ on $ \eta $ is greater than $ V_{1} $ on $ Q_1 $; recall Lemma~\ref{lem:P}) and, if it does, on the initial and final points. So, fixing two points $ q_1, q_2 \in Q_1 $, the global minimum of $ T_{\hat{\mathcal{N}}^*} $ is the three-segment trajectory satisfying the above conditions, if it exists, or the usual reflected trajectory, if it does not. This global comparison should, of course, be benchmarked against the traveltime for the direct straight line between $ q_{1} $ and $ q_{2} $. This will be done in \S~\ref{subsec:global} below.

\begin{rem}
\label{rem:huygens}
At first glance, these new reflected trajectories might only seem relevant in situations when you can willingly change the directions midway at points of $ \eta $. Obviously, a light ray refracted at the critical angle will continue along $ \eta $ indefinitely, without returning to $ Q_1 $. When computing the whole propagation of a wave though, it is crucial to take into account all the possible trajectories. Indeed, Huygens' principle states that each point of the wavefront can be regarded as an independent source that shoots wave trajectories in all directions. Among all of them, those that globally minimize the traveltime are the ones that keep providing points in the wavefront at later instants of times. So, a curve such as $ \gamma $ in Figure~\ref{fig:new_reflection} can indeed represent a realistic wave trajectory that contributes to the formation of the wavefront (see \S~\ref{subsec:wavefront}).
\end{rem}

\begin{figure}
\centering
\includegraphics[width=1\textwidth]{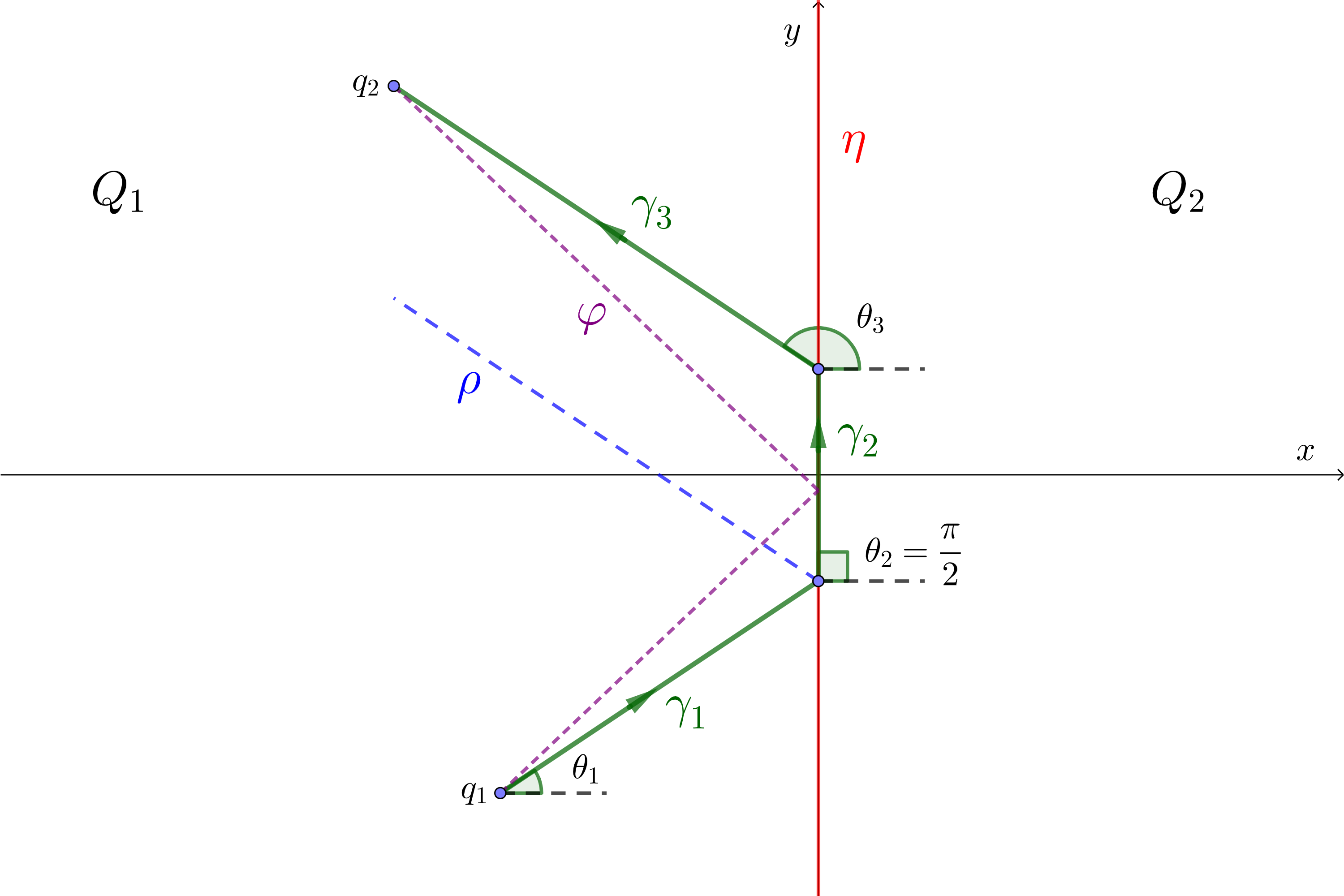}
\caption{There are situations where moving along $ \gamma $ is faster than following the usual reflected trajectory $ \varphi $ or even the straight line between two fixed points $ q_1, q_2 \in Q_1 $. This can only happen if $ \gamma_2 $ is traveled with velocity $ V_2(\frac{\pi}{2}) $ along $ \eta $ and if $ \theta_1 $ and $ \theta_3 $ are, respectively, the critical and reflection angles of $ \gamma_1 $. In particular, $ \gamma_3 $ must be parallel to the reflection $ \rho $ of $ \gamma_1 $.}
\label{fig:new_reflection}
\end{figure}

\subsection{Globally time-minimizing trajectories}
\label{subsec:global}
Taking a step further, we can now remove the restriction that the curves must touch $ \eta $.

\begin{defi}
Fixing any two points $ q_1, q_2 \in Q $, let $ \mathcal{M} $ be the set of all (regular) piecewise smooth curves from $ q_1 $ to $ q_2 $. We say that $ \gamma \in \mathcal{M} $ {\em globally minimizes the traveltime} if $ T[\gamma] \leq T[\varphi] $ for all $ \varphi \in \mathcal{M} $.
\end{defi}

If $ q_1 \in Q_1, q_2 \in Q_2 $, then it is clear that $ \mathcal{M} = \hat{\mathcal{N}} $, so the refracted trajectories always globally minimize the traveltime. This includes the critical trajectories along $ \eta $, as stated in the following result.

\begin{prop}
\label{prop:critical}
If $ \gamma: [0,t_0] \rightarrow Q $ is a critical trajectory (associated with one of the critical angles $ \theta_c^\pm $), then $ \gamma $ globally minimizes the traveltime between $ \gamma(0) \in Q_1 $ and $ \gamma(t_0) \in \eta $.
\end{prop}
\begin{proof}
Given $ \gamma(t_0) = (0,y_0) \in \eta $, suppose there exists another curve $ \xi_1 \in \mathcal{M} $ that arrives earlier at this point from $ \gamma(0) \in Q_1 $. Consider now $ (\delta,y_0) \in Q_2 $ with $ \delta > 0 $ and let $ \varphi_{\delta} $ be the unique refracted trajectory from $ \gamma(0) $ to $ (\delta,y_0) $, which globally minimizes the traveltime. The smoothness of the raypath parameter given by \eqref{eq:raypath} ensures that
\begin{equation*}
\lim_{\delta \rightarrow 0} T[\varphi_{\delta}] = T[\gamma] > T[\xi_1].
\end{equation*}
So, for a sufficiently small $ \delta > 0 $, the straight line $ \xi_2 $ from $ (0,y_0) $ to $ (\delta,y_0) $ satisfies that $ T[\varphi_{\delta}] > T[\xi_1]+T[\xi_2] $, which means that we have built a path in $ \mathcal{M} $ faster than $ \varphi_{\delta} $, arriving at a contradiction.
\end{proof}

If $ q_1, q_2 \in Q_1 $ though, the global minimum of $ T_{\hat{\mathcal{N}}^*} $ competes against the straight line in $ Q_1 $ between $ q_1 $ and $ q_2 $. It is clear that the straight line is always faster than the usual reflection, due to the triangle inequality for the metric $ F_{1} $ in $ Q_{1} $, but in some situations the three-segment reflected trajectory can globally minimize the traveltime between $ q_{1} $ and $ q_{2} $, thus creating a point on the cut locus from $ q_{1} $ -- we elaborate further on this issue in \S~\ref{subsec:cutlocus} below. We first illustrate the setting for these possibilities next with an example.

\subsection{Anisotropic example}
\label{subsec:example}
Let us consider the following velocity functions in $ \mathds{R}^{2} $:
\begin{equation*}
V_j(\theta) = \frac{a_j(1-\varepsilon_j^2)}{1-\varepsilon_j \cos(\theta-\phi_j)}, \qquad j=1,2,
\end{equation*}
which define ellipses of eccentricity $ \varepsilon_j $, respectively, centered at one of their foci, with semi-major axis $ a_j $ oriented in the direction $ \phi_j $ (with respect to the $ \dot{x} $-axis) and such that $ \phi_j = 0 $ points to their other respective focus point. Pecisely, $ V_j(\theta) $ is the (Euclidean) length of the vector from the centered focus to the ellipse in the direction $ \theta $, and this ellipse becomes the indicatrix of the Finsler metric\footnote{This type of Finsler metrics are used to model the spread of wildfires (see e.g. \cite{F,JPS2}).}
\begin{equation*}
F_j(\dot{x},\dot{y}) = \frac{\sqrt{\dot{x}^2+\dot{y}^2}}{a_j(1-\varepsilon_j^2)} \left[ 1-\varepsilon_j \cos\left(\arctan\left(\frac{\dot{y}}{\dot{x}}\right)-\phi_j\right) \right], \qquad j=1,2.
\end{equation*}
One can easily verify that
\begin{equation*}
\frac{\partial F_j}{\partial \dot{y}}(\theta) = \frac{\sin{\theta}-\varepsilon_j\sin{\phi_j}}{a_j(1-\varepsilon_j^2)}, \qquad j=1,2,
\end{equation*}
so \eqref{eq:ref_R2} becomes
\begin{equation*}
\frac{\partial F_1}{\partial\dot{y}}(\theta_k) = \frac{\partial F_2}{\partial\dot{y}}(\theta_2) \quad \Rightarrow \quad \frac{\sin{\theta_k}-\varepsilon_1\sin{\phi_1}}{a_1(1-\varepsilon_1^2)} = \frac{\sin{\theta_2}-\varepsilon_2\sin{\phi_2}}{a_2(1-\varepsilon_2^2)},
\end{equation*}
for $ k = 1,3 $. In particular, in order to find the three-segment reflected trajectory, $ \theta_1 $ must be the (critical) angle such that $ \theta_2 = \pm \frac{\pi}{2} $, i.e.
\begin{equation}
\label{eq:crit_angle}
\theta_1 = \arcsin\left( \frac{a_1(1-\varepsilon_1^2)}{a_2(1-\varepsilon_2^2)}(\pm 1-\varepsilon_2 \sin{\phi_2}) + \varepsilon_1 \sin{\phi_1} \right),
\end{equation}
and $ \theta_3 $ must satisfy the same equation but pointing to $ Q_1 $.\footnote{Even though the velocities are anisotropic, the law of reflection in this example is the same as in the case with Euclidean metrics (see Example~\ref{ex:classical_reflection}). However, this is not true in general.}

Now, suppose we want to go from $ q_1 = (-x_0,-y_0) $ to $ q_2 = (-x_0,y_0) $, for some $ x_0, y_0 > 0 $, and let us set $ a_1=a_2=1 $, $ \varepsilon_1=\varepsilon_2=\frac{1}{2} $, $ \phi_1 = 0 $ and $ \phi_2 = \frac{\pi}{2} $. Then \eqref{eq:crit_angle} becomes $ \theta_1 = \arcsin\left(\pm 1-\frac{1}{2}\right) $, which does not have solution for the $ - $ sign (corresponding to $ \theta_2 = -\frac{\pi}{2} $), but it does for the $ + $ sign (corresponding to $ \theta_2 = \frac{\pi}{2} $). Namely, there is only one critical angle $ \theta_1 = \theta_c^+ = \arcsin{\frac{1}{2}} = \frac{\pi}{6} $ and thus, $ \theta_3 = \pi-\theta_1 = \frac{5\pi}{6} $. We can now compute the intersection points with $ \eta $ to see if the three-segment reflection exists. Since we can completely determine $ \gamma_1 $ and $ \gamma_3 $ (we know $ q_1, q_2 $ and $ \theta_1, \theta_3 $), it is easy to show that
\begin{equation*}
\gamma_1 \cap \eta = \left( 0,-y_0+\frac{x_0}{\sqrt{3}} \right), \qquad \gamma_3 \cap \eta = \left( 0,y_0-\frac{x_0}{\sqrt{3}} \right).
\end{equation*}
We conclude (see Figure~\ref{fig:example}):
\begin{itemize}
\item If $ y_0 > \frac{x_0}{\sqrt{3}} $, then the three-segment curve exists and it is the global minimum of $ T_{\hat{\mathcal{N}}^*} $.

\item If $ y_0 = \frac{x_0}{\sqrt{3}} $, then $ \gamma_2 $ reduces to the point $ (0,0) $ and $ \gamma_3 $ becomes the reflection of $ \gamma_1 $ corresponding to the critical angle.

\item If $ y_0 < \frac{x_0}{\sqrt{3}} $, then you cannot reach $ q_2 $ through a three-segment curve with the given angles $ \theta_1, \theta_2, \theta_3 $. In this case, the global minimum of $ T_{\hat{\mathcal{N}}^*} $ is the usual reflected trajectory, with an angle of incidence smaller than the critical angle.
\end{itemize}

Finally, we can find the trajectory that globally minimizes the traveltime. This can be done simply by comparing the traveltimes of the straight line $ \xi $ and the three-segment reflection $ \gamma $ in the case $ y_0 > \frac{x_0}{\sqrt{3}} $ (in any other case the straight line is always the fastest trajectory). Using the endpoints and direction of each segment, they can be parametrized as follows:
\begin{equation*}
\begin{split}
& \xi(t) = (-x_0,-y_0)+t(0,2), \qquad t \in [0,y_0], \\
& \gamma_1(t) = (-x_0,-y_0)+t(\sqrt{3},1), \qquad t \in \left[0,\frac{x_0}{\sqrt{3}}\right], \\
& \gamma_2(t) = \left(0,-y_0+\frac{x_0}{\sqrt{3}}\right)+t(0,2), \qquad t \in \left[0,y_0-\frac{x_0}{\sqrt{3}}\right], \\
& \gamma_3(t) = \left(0,y_0-\frac{x_0}{\sqrt{3}}\right)+t(-\sqrt{3},1), \qquad t \in \left[0,\frac{x_0}{\sqrt{3}}\right]
\end{split}
\end{equation*}
(note that this is not the time-parametrization). Therefore
\begin{equation*}
T[\xi] = \int_0^{y_0} F_1(0,2)dt = \frac{4}{3} y_0,
\end{equation*}
\begin{equation*}
\begin{split}
T[\gamma] = & \int_0^{\frac{x_0}{\sqrt{3}}} F_1(\sqrt{3},1)dt + \int_0^{y_0-\frac{x_0}{\sqrt{3}}} F_2(0,2)dt +\int_0^{\frac{x_0}{\sqrt{3}}} F_1(-\sqrt{3},1)dt = \\
= & \frac{4}{3}\left(\frac{x_0}{\sqrt{3}}(4-\sqrt{3})+1\right),
\end{split}
\end{equation*}
and we conclude (see Figure~\ref{fig:example}):
\begin{itemize}
\item If $ y_0 = \frac{x_0}{\sqrt{3}}(4-\sqrt{3})+1 $, then $ \gamma $ and $ \xi $ reach $ q_2 $ at the same time and both globally minimize the traveltime. In this case $ q_{2} $ belongs to the cut locus of $ q_{1} $ -- see \S~\ref{subsec:cutlocus} below.

\item If $ y_0 > \frac{x_0}{\sqrt{3}}(4-\sqrt{3})+1 $, only $ \gamma $ globally minimizes the traveltime.

\item If $ y_0 < \frac{x_0}{\sqrt{3}}(4-\sqrt{3})+1 $, only $ \xi $ globally minimizes the traveltime.
\end{itemize}

\begin{figure}
\centering
\includegraphics[width=1\textwidth]{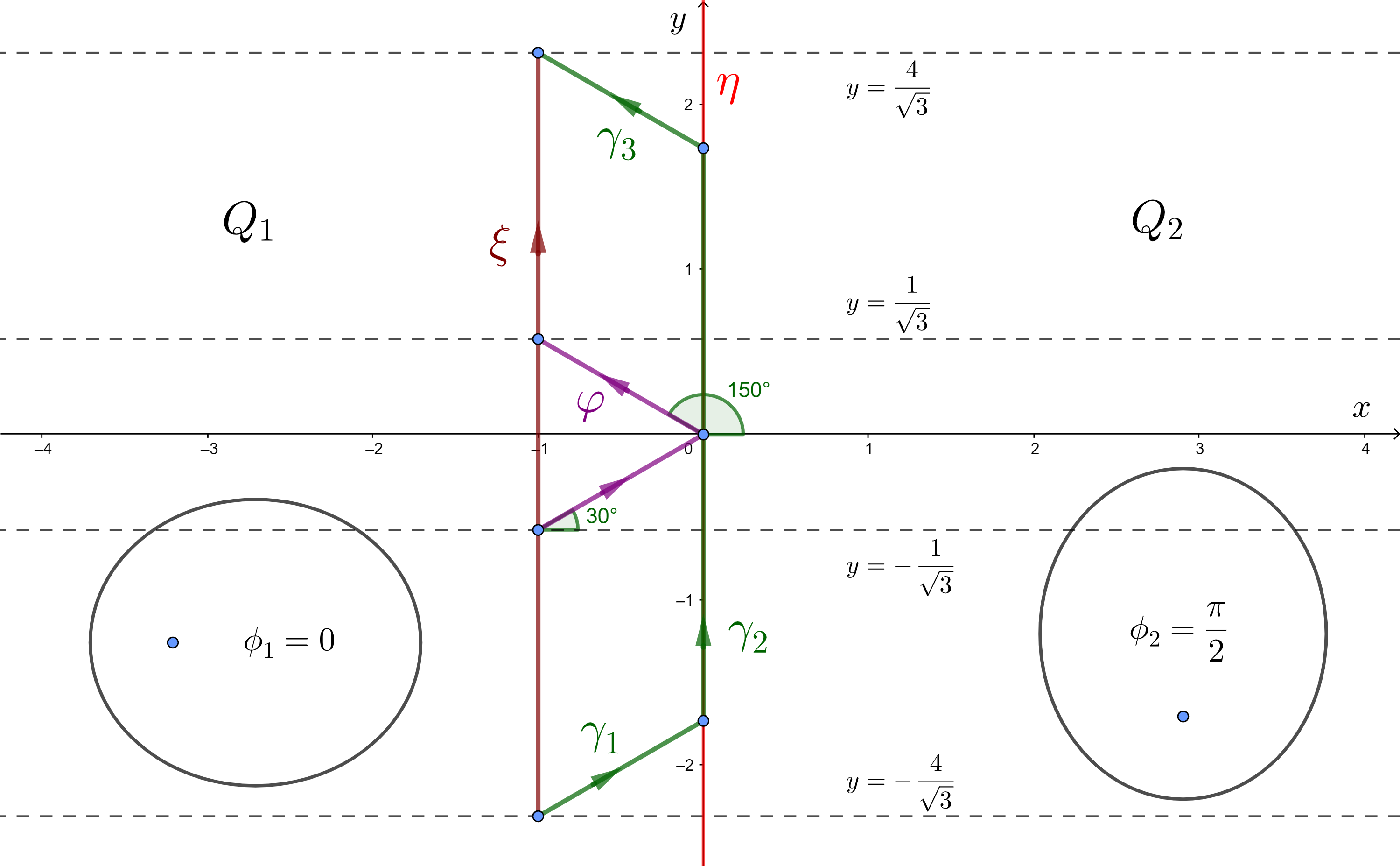}
\caption{In this example, the velocity at each point is given by a focus-centered ellipse with orientation $ \phi_1=0 $ in $ Q_1 $ and $ \phi_2=\frac{\pi}{2} $ in $ Q_2 $, which favors the travel along $ \eta $. Fixing $ x_0 = 1 $ and always considering the travel from $ q_1 = (-1,-y_0) $ to its symmetrical $ q_2 = (-1,y_0) $, there are three different regions for $ y_0 $: if $ 0<y_0<\frac{1}{\sqrt{3}} $, the fastest trajectory via $ \eta $ is the usual reflection passing through $ (0,0) $, with an angle of incidence smaller than the critical one; if $ \frac{1}{\sqrt{3}}<y_0<\frac{4}{\sqrt{3}} $, then the three-segment reflection becomes faster than the usual reflected trajectory, but still in these first two cases $ \xi $ is the fastest path; when $ y_0 > \frac{4}{\sqrt{3}} $, the three-segment reflection becomes even faster than $ \xi $. The limit cases are depicted explicitly: $ \varphi $ when $ y_0 = \frac{1}{\sqrt{3}} $ and $ \gamma $ when $ y_0 = \frac{4}{\sqrt{3}} $.}
\label{fig:example}
\end{figure}

\subsection{Construction of the wavefront}
\label{subsec:wavefront}
Fixing a time $ t_0 > 0 $, consider all the time-parametrized trajectories $ \gamma $ that globally minimize the traveltime from $ q_1 \in Q_1 $ to $ \gamma(t_0) $ (and thus, to any $ \gamma(t) $ with $ t \in (0,t_0] $). They provide the outermost points that can be reached by a wave spreading from $ q_1 $ during the time lapse $ t_0 $, i.e. each one of these trajectories provide a point in the wavefront.

Throughout this section we have seen that only three types of curves can globally minimize the traveltime in the classical Finslerian setting: the straight lines in $ Q_1 $, the refracted trajectories (crossing $ \eta $ once or remaining on it) and the three-segment reflected trajectories. Separately, they provide the {\em standard}, {\em refracted} and {\em reflected} {\em wavefronts}, respectively, which compete against each other to generate the actual wavefront.

Let $ \tau_\eta $ and $ \tau^\pm $ be the times when the standard wavefront and the critical trajectories reach $ \eta $, respectively ($ \tau_\eta $ coincides with the $ F_1 $-distance from $ q_1 $ to $ \eta $). The refracted wavefront only comes into existence from $ \tau_\eta $, and the reflected ones from $ \tau^\pm $.

\begin{lemma}
\label{lem:wavefronts}
Fixing any $ q_1 \in Q_1 $ and a time $ t_0 > \tau^\pm $:
\begin{enumerate}[(i)]
\item The standard wavefront $ \tilde{\Gamma}_{t_0} $ at $ t=t_0 $ is given by the indicatrix of $ F_1 $ scaled up by $ t_0 $, which is equivalent to
\begin{equation}
\label{eq:tilde_Gamma}
\tilde{\Gamma}_{t_0} = \{ p \in \mathds{R}^2: F_1(p-q_1) = t_0 \},
\end{equation}
and $ \tilde{\Gamma}_{t_0} $ intersects $ \eta $ at two points.

\item There exists one reflected wavefront $ \hat{\Gamma}^{\pm}_{t_0} $ at $ t=t_0 $ for each critical angle $ \theta^{\pm}_c $. If $ \varphi^{\pm} $ is the usual reflection at $ \theta^{\pm}_c $ and $ \gamma^{\pm} $ is the corresponding critical trajectory along $ \eta $ (both time-parametrized), then $ \hat{\Gamma}^{\pm}_{t_0} $ is the straight line joining $ \gamma^{\pm}(t_0) $ and $ \varphi^{\pm}(t_0) $.

\item The intersection between $ \tilde{\Gamma}_{t_0} $ and each $ \hat{\Gamma}^{\pm}_{t_0} $ is a unique point, provided that $ \hat{\Gamma}^{\pm}_{t_0} $ exists.
\end{enumerate}
\end{lemma}
\begin{proof}
For the claim $ (i) $, if we consider $ F_1 $-unit speed straight lines $ \xi $, note that they travel a distance $ t_0\|\xi'(t)\|_g $ with $ \|\xi'(t)\|_g $ constant and $ F_1(\xi'(t)) = 1 $, so $ t_0\Sigma_1 $ provides the distance traveled by $ \xi $ in all directions, being $ \Sigma_1 $ the indicatrix of $ F_1 $. Centered at $ q_1 $, $ t_0\Sigma_1 $ gives the points \eqref{eq:tilde_Gamma}. Also, as $ t_0 > \tau^\pm \geq \tau_\eta $, the strong convexity of $ \Sigma_1 $ ensures that $ \tilde{\Gamma}_{t_0} $ intersects $ \eta $ at exactly two points.

We will prove $ (ii) $ and $ (iii) $ for $ \hat{\Gamma}^+_{t_0} $, being the proof for $ \hat{\Gamma}^-_{t_0} $ completely analogous. For $ (ii) $, let $ \rho $ be any (time-parametrized) three-segment reflected trajectory with angle of incidence $ \theta_1 = \theta_c^+ $. Assume $ \rho $ reaches $ \eta $ at time $ \tau_1 > 0 $ (the same as $ \varphi^+ $ and $ \gamma^+ $) and leaves it at $ \tau_2 $, with $ \tau_1 < \tau_2 < t_0 $. Splitting the curves in their different segments, note that $ \rho_1 = \varphi^+_1 = \gamma^+_1 $, $ \rho_2 $ is part of $ \gamma^+_2 $ and $ \rho_3 $ is parallel to $ \varphi^+_2 $. Thus
\begin{equation*}
\frac{\|\rho_3'(t)\|_g(t_0-\tau_2)}{\|(\gamma_2^+)'(t)\|_g(t_0-\tau_2)} = \frac{\|\rho_3'(t)\|_g}{\|(\gamma_2^+)'(t)\|_g} = \frac{\|(\varphi_2^+)'(t)\|_g}{\|(\gamma_2^+)'(t)\|_g} = \frac{\|(\varphi_2^+)'(t)\|_g(t_0-\tau_1)}{\|(\gamma_2^+)'(t)\|_g(t_0-\tau_1)},
\end{equation*}
which means, by Thales' theorem, that $ \rho(t_0) $ is in the segment $ \hat{\Gamma}^+_{t_0} $ joining $ \gamma^+(t_0) $ and $ \varphi^+(t_0) $. This also ensures that given any point $ p \in \hat{\Gamma}^+_{t_0} $, there exists a three-segment reflected trajectory from $ q_1 $ to $ p $. So $ \hat{\Gamma}^+_{t_0} $ is the reflected wavefront.

For the claim $ (iii) $, note that $ \gamma^+ $ always globally minimizes the traveltime by Proposition~\ref{prop:critical}, but $ \varphi^+ $ never does. In particular, this means that
\begin{equation*}
F_1(\varphi^+(t_0)-q_1)< t_0 < F_1(\gamma^+(t_0)-q_1),
\end{equation*}
so $ \varphi^+(t_0) $ lies inside $ \tilde{\Gamma}_{t_0} $, whereas $ \gamma^+(t_0) $ is outside. Given the fact that $ \hat{\Gamma}^+_{t_0} $ is the straight line joining both points and $ \tilde{\Gamma}_{t_0} $ is strongly convex, there is a unique intersection point.

%first that $ \tilde{\Gamma}_{t_0} \cap \hat{\Gamma}^+_{t_0} $ has at most two points, given the strong convexity of $ \tilde{\Gamma}_{t_0} $. Now, 

%so there must be at least one point $ p \in \hat{\Gamma}^+_{t_0} $ where $ F_1(p - q_1) = t_0 $ and thus, $ p \in \tilde{\Gamma}_{t_0} \cap \hat{\Gamma}^+_{t_0} $. If there was another intersection point, it would mean that $ \varphi^+ $ globally minimizes the traveltime, which is a contradiction.
\end{proof}

Note that $ \tilde{\Gamma}_{t_0} $ and $ \hat{\Gamma}^{\pm}_{t_0} $ can be explicitly parametrized as (smooth) curves in $ \mathds{R}^2 $ as follows:
\begin{equation}
\label{eq:parametrization}
\begin{split}
& \tilde{\Gamma}_{t_0}(\theta) = t_0 V_1(\theta)(\cos\theta,\sin\theta)+q_1, \qquad \theta \in (-\pi,\pi], \\
& \hat{\Gamma}^{\pm}_{t_0}(s) = \gamma^{\pm}(t_0)+s(\varphi^{\pm}(t_0)-\gamma^{\pm}(t_0)), \qquad s \in [0,1],
\end{split}
\end{equation}
and Lemma~\ref{lem:wavefronts} ensures there are exactly two angles $ \theta_{\eta}^{\pm} $ such that $ \tilde{\Gamma}_{t_0}(\theta_{\eta}^{\pm}) \in \eta $ and, if $ \theta_c^{\pm} $ exists, we can find $ s_0^{\pm} $ and $ \theta_0^{\pm} $ such that $ \tilde{\Gamma}_{t_0}(\theta_0^{\pm}) = \hat{\Gamma}_{t_0}(s_0^{\pm}) $. Note that $ \theta_{\eta}^{\pm}$, $ s_0^{\pm} $ and $ \theta_0^{\pm} $ depend on $ t_0 $.

\begin{thm}
\label{th:wavefront}
Assume for simplicity that $ \theta_c^+ $ exists but $ \theta_c^- $ does not (analogous results are obtained in any other case). Then the actual wavefront $ \Gamma_{t_0} $ from $ q_1 \in Q_1 $ at $ t=t_0 $ is a piecewise smooth closed curve composed of the following parts:
\begin{itemize}
\item If $ t_0 \in [0,\tau_\eta] $, then $ \Gamma_{t_0} $ is the whole standard wavefront at $ t=t_0 $ (see Figure~\ref{fig:wave_a}).

\item If $ t_0 \in (\tau_\eta,\tau^+] $, then $ \Gamma_{t_0} = \Gamma_1 \cup \Gamma_2 $, where
\begin{equation*}
\Gamma_1 = \{ \tilde{\Gamma}_{t_0}(\theta) \in \mathds{R}^2: \theta \in (-\pi,\theta_{\eta}^-] \cup [\theta_\eta^+,\pi] \}
\end{equation*}
and $ \Gamma_2 $ is the whole refracted wavefront at $ t=t_0 $ (see Figure~\ref{fig:wave_b}).

\item If $ t_0 > \tau^+ $, then $ \Gamma_{t_0} = \Gamma_1 \cup \Gamma_2 \cup \Gamma_3 $, where
\begin{equation*}
\begin{split}
& \Gamma_1 := \{ \tilde{\Gamma}_{t_0}(\theta) \in \mathds{R}^2: \theta \in (-\pi,\theta_{\eta}^-] \cup [\theta_0^+,\pi] \}, \\
& \Gamma_3 := \{ \hat{\Gamma}^+_{t_0}(s) \in \mathds{R}^2: s \in [0,s_0^+] \}
\end{split}
\end{equation*}
and $ \Gamma_2 $ is the whole refracted wavefront at $ t=t_0 $ (see Figures~\ref{fig:wave_c} and \ref{fig:wave_d}).
\end{itemize}
\end{thm}
\begin{proof}
Note first that $ \Gamma_1 $ and $ \Gamma_3 $ are always smooth curves by \eqref{eq:parametrization} and so is $ \Gamma_2 $ by Theorem~\ref{th:existence_refraction} and the smoothness of the raypath parameter \eqref{eq:raypath}. Now, if $ t_0 \in [0,\tau_\eta] $, only the standard wavefront exists. If $ t_0 \in (\tau_\eta,\tau^+] $, the reflected wavefront still does not exist, so the part of the standard wavefront $ \Gamma_1 $ in $ Q_1 $ globally minimizes the traveltime, and so does the whole refracted wavefront $ \Gamma_2 $ (recall the discussion in \S~\ref{subsec:global}). Also, the endpoints of $ \Gamma_1 $ and $ \Gamma_2 $ coincide, which ensures that $ \Gamma_{t_0} $ is a piecewise smooth closed curve. Finally, if $ t_0 > \tau^+ $, it is clear from \S~\ref{subsec:global} and Lemma~\ref{lem:wavefronts} that each part $ \Gamma_1, \Gamma_2, \Gamma_3 $ globally minimizes the traveltime. In this case, the endpoints of $ \Gamma_1 $ are $ \Gamma_2(\theta_{\eta}^-) $ and $ \Gamma_3(0) $, and $ \Gamma_2(\theta_0^+) = \Gamma_3(s_0^+) $. So $ \Gamma_{t_0} $ is still a piecewise smooth closed curve.
\end{proof}

\begin{figure}
\centering
\subfigure[$ t_0 < \tau_{\eta} $]{\label{fig:wave_a}\includegraphics[width=0.49\textwidth]{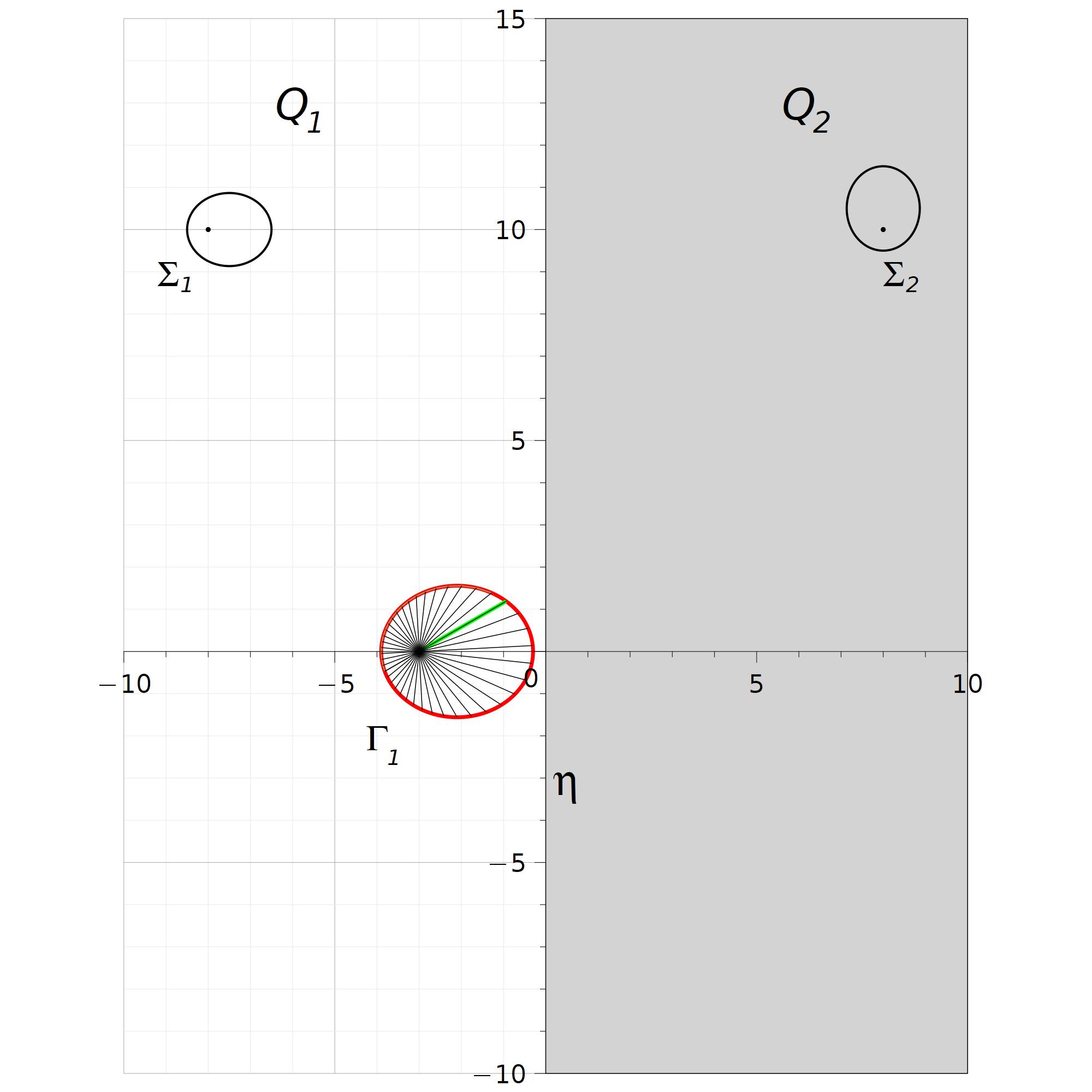}}
\subfigure[$ t_0 = \tau^+ $]{\label{fig:wave_b}\includegraphics[width=0.49\textwidth]{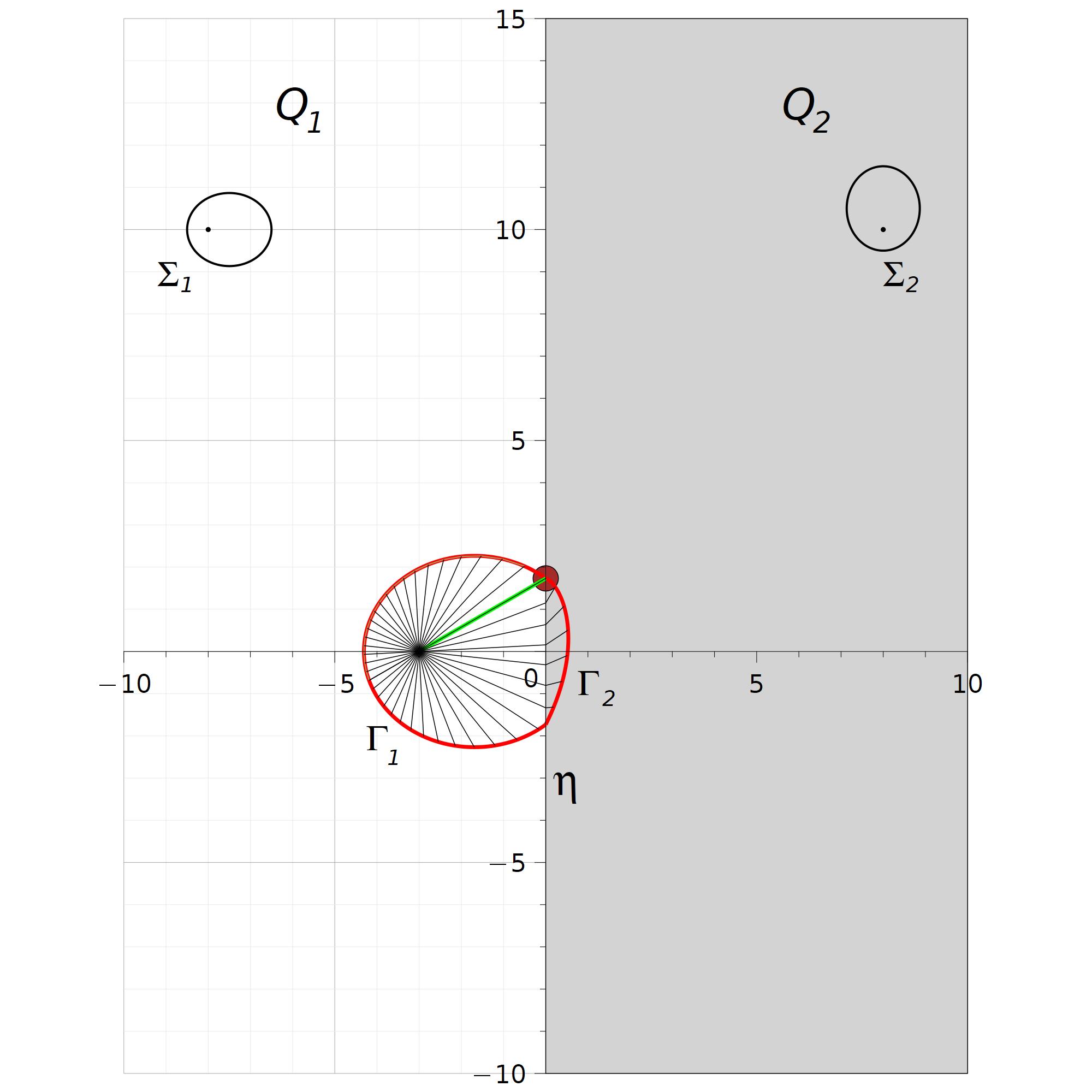}}
\subfigure[$ t_0 > \tau^+ $]{\label{fig:wave_c}\includegraphics[width=0.49\textwidth]{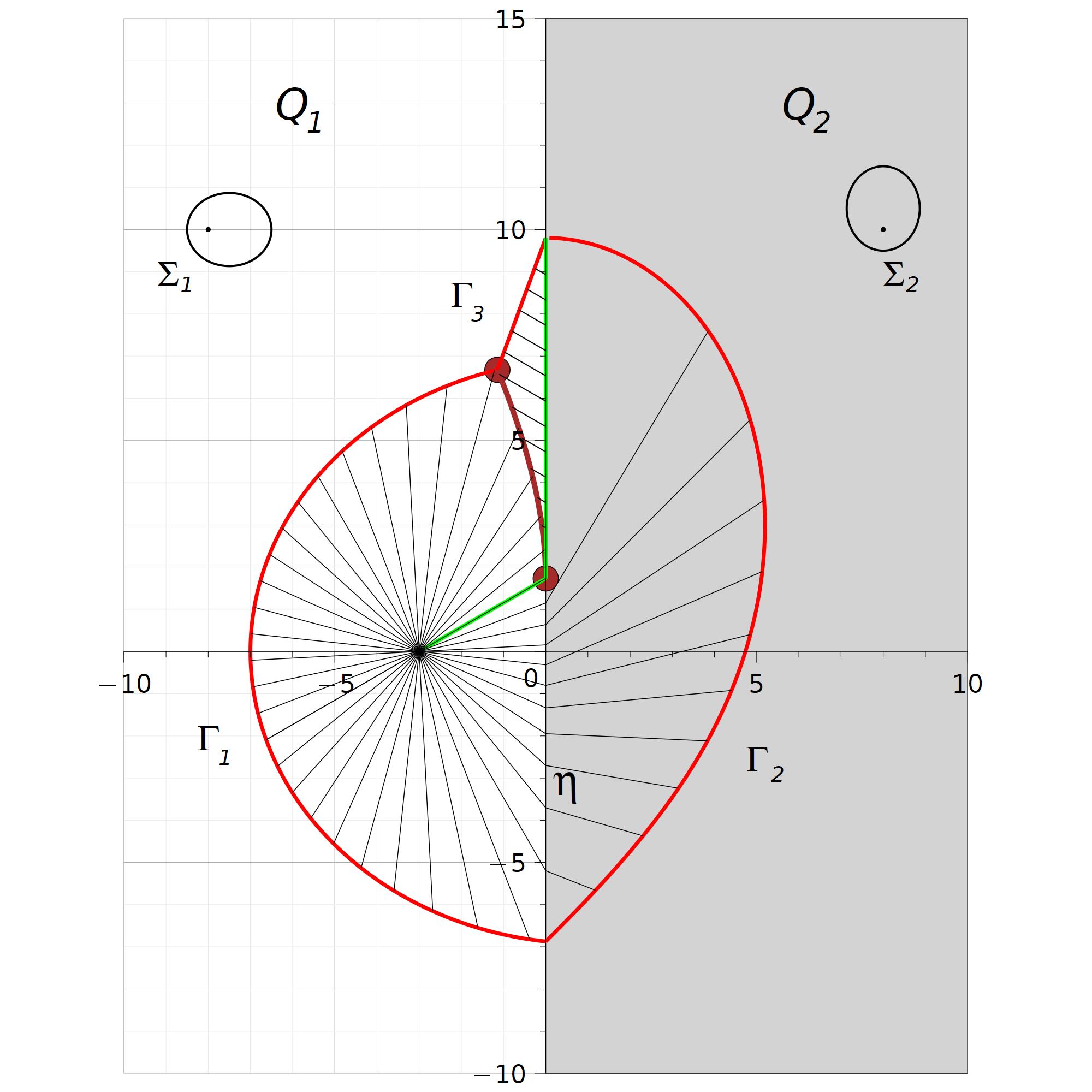}}
\subfigure[$ t_0 > \tau^+ $]{\label{fig:wave_d}\includegraphics[width=0.49\textwidth]{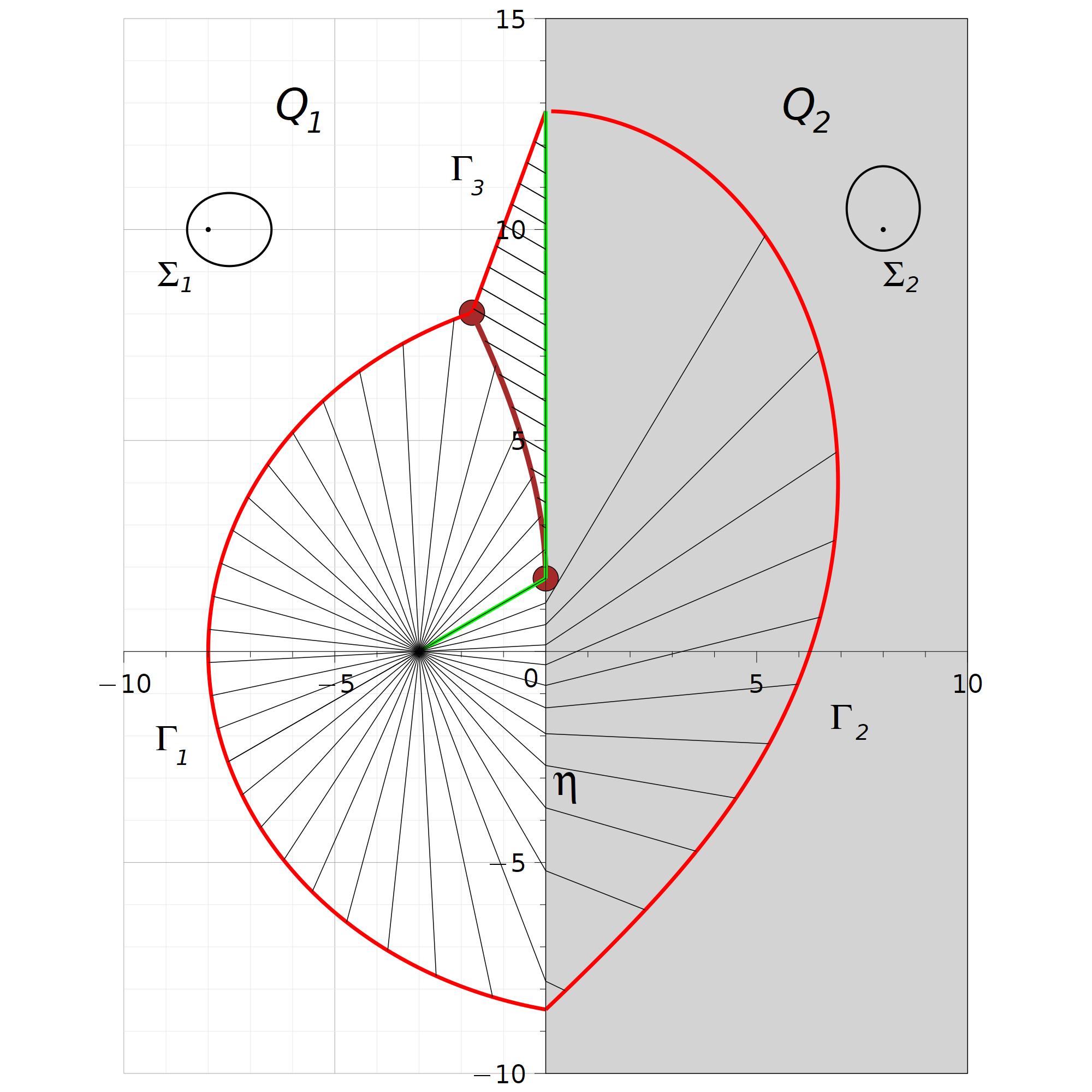}}
\caption{Evolution of the wavefront over time. The cut locus in (c) and (d) given by the intersection of $ \Gamma_1 $ and $ \Gamma_3 $ for $ t \in (\tau^+,t_0] $ develops into $ Q_{1} $ from the point of first encounter of the expanding ellipse with $ \eta $ in the critical direction, as shown in (b). The Finsler metrics are the same as in \S~\ref{subsec:example}.}
\label{fig:wavefront}
\end{figure}

\subsection{The cut locus}
\label{subsec:cutlocus}
Fixing any $ q \in Q $, let us define the following curves:
\begin{equation*}
\begin{array}{cccc}
\sigma^\pm \colon & (\tau^\pm,\infty) & \longrightarrow & \mathds{R}^2 \\
& t & \longmapsto & \sigma^\pm(t) = \tilde{\Gamma}_t \cap \hat{\Gamma}^\pm_t,
\end{array}
\end{equation*}
where $ \tilde{\Gamma}_t, \hat{\Gamma}^\pm_t $ are the standard and reflected wavefronts from $ q $ at time $ t $ and $ \tau^\pm $ is the time at which the critical trajectory $ \gamma^\pm $ reaches $ \eta $. Provided that the critical angles exist, note that each $ \sigma^\pm $ is well defined, as $ \tau^\pm $ is the time from which the reflected wavefront begins to exist and $ \tilde{\Gamma}_t \cap \hat{\Gamma}^\pm_t $ is just one point by Lemma~\ref{lem:wavefronts} (see Figures~\ref{fig:wave_c} and \ref{fig:wave_d}).

By definition, each point of these curves is reached at the same time by two different globally minimizing trajectories from $ q $, constituting what is known as the {\em cut locus} of $ q $:
\begin{equation*}
\mathcal{C}_{q} := \text{Im}(\sigma^+) \cup \text{Im}(\sigma^-).
\end{equation*}

\begin{prop}
Fix $ q \in Q $ and let $ \mathcal{M}_q $ be the set of all piecewise smooth curves $ \gamma: [0,\infty) \rightarrow Q $ (and thus, future inextendible) with $ \gamma(0) = q $. If we define the map $ t_q: \mathcal{M}_q \rightarrow (0,\infty] $ by
\begin{equation*}
t_q(\gamma) := \textup{sup}\{ t \in \mathds{R}: \gamma|_{[0,t]} \text{ globally minimizes the traveltime} \},
\end{equation*}
then the cut locus of $ q $ is
\begin{equation}
\label{eq:cut_locus}
\mathcal{C}_q = \{ \gamma(t_q(\gamma)) \in \mathds{R}^2: \gamma \in \mathcal{M}_q, 0 < t_q(\gamma) < \infty \}.
\end{equation}
\end{prop}
\begin{proof}
Let $ p = \tilde{\Gamma}_{t_0} \cap \hat{\Gamma}^\pm_{t_0} \in \mathcal{C}_q $ for any $ t_0 > \tau^\pm $. There exist two (time-parametrized) curves $ \xi, \gamma \in \mathcal{M}_q $ such that $ \xi(t_0) = \gamma(t_0) = p $ and $ \xi_{[0,t_0]}, \gamma_{[0,t_0]} $ globally minimize the traveltime. Suppose $ \gamma|_{[0,t_0+\delta]} $ is still globally time-minimizing for some $ \delta > 0 $. Then the concatenation of $ \xi_{[0,t_0]} $ with $ \gamma|_{[t_0,t_0+\delta]} $ also globally minimizes the traveltime. But if we concatenate $ \xi_{[0,t_0-\delta]} $ with the straight line joining $ \xi(t_0-\delta) $ and $ \gamma(t_0+\delta) $ we reduce the traveltime by the triangle inequality (note that $ \gamma $ and $ \xi $ cannot be parallel), which is a contradiction. So $ t_0 = t_q(\gamma) $ and $ p $ is in the set \eqref{eq:cut_locus}.

Conversely, let $ p = \gamma(t_q(\gamma)) $ for some $ \gamma \in \mathcal{M}_q $ with $ 0 < t_q(\gamma) < \infty $. Since $ \gamma|_{[0,t_q(\gamma)]} $ globally minimizes the traveltime, $ \gamma $ is either a refracted trajectory, a straight line or a three-segment reflected trajectory. Now, if it were a refracted trajectory, then $ t_q(\gamma) = \infty $, which is a contradiction, so assume without loss of generality that $ \gamma(t_q(\gamma)) \in \hat{\Gamma}^+_{t_q(\gamma)} $. For each $ \delta > 0 $, the global time-minimizing curve from $ q $ to $ \gamma(t_q(\gamma)+\delta) $ must be the straight line joining both points (it cannot be a tree-segment reflected trajectory because its last segment would not intersect $ \gamma $) and in the limit $ \delta \rightarrow 0 $, we find the (time-parametrized) straight line $ \xi $ from $ q $ to $  \gamma(t_q(\gamma)) $ satisfying $ \xi(t_q(\gamma)) \in \tilde{\Gamma}_{t_q(\gamma)} $. So $ p = \gamma(t_q(\gamma)) \in \tilde{\Gamma}_{t_q(\gamma)} \cap \hat{\Gamma}^+_{t_q(\gamma)} \in \mathcal{C}_q $.
\end{proof}

Essentially, any wavefront trajectory crossing one of the curves $ \sigma^\pm $ loses the property of being globally time-minimizing, no longer providing points in the wavefront afterwards. In practice, $ \sigma^\pm $ represent the merging regions of two wavefronts coming from different directions, becoming extremely relevant in some situations. When dealing with wildfires, for instance, these curves are danger zones for the firefighters, as they can easily become trapped by the fire and the convergence of several fire trajectories leads to the corresponding increase in the heat intensity (see also the discussion in \cite{JPS2}).

\section*{Acknowledgments}
The authors warmly thank the anonymous referee, whose accurate comments helped improve the manuscript.

This work is the result of a research stay of EPR at DTU Compute, partially supported by The International Doctoral School of the University of Murcia (EIDUM) and by Ayudas para la Formaci\'{o}n de Profesorado Universitario (FPU) from the Spanish Government.

EPR was supported by the project PID2021-124157NB-I00, funded by MCIN/AEI/10.13039/501100011033/ and by "ERDF A way of making Europe", and he was also supported by Ayudas a proyectos para el desarrollo de investigaci\'{o}n cient\'{i}fica y t\'{e}cnica por grupos competitivos (Comunidad Aut\'{o}noma de la Regi\'{o}n de Murcia), included in the Programa Regional de Fomento de la Investigaci\'{o}n Cient\'{i}fica y T\'{e}cnica (Plan de Actuaci\'{o}n 2022) of the Fundaci\'{o}n S\'{e}neca-Agencia de Ciencia y Tecnolog\'{i}a de la Regi\'{o}n de Murcia, REF. 21899/PI/22.

\end{document}